\documentclass[11pt,reqno]{amsart}
\usepackage{amsthm,amssymb,amsmath}
\usepackage{textcomp}
\usepackage{enumitem}

\usepackage[T1]{fontenc}
\usepackage{lmodern}

\RequirePackage{fix-cm}

\newtheorem{theorem}{Theorem}
\newtheorem*{theorem*}{Theorem}
\newtheorem{lemma}{Lemma}
\newtheorem{corollary}{Corollary}
\newtheorem{proposition}{Proposition}

\theoremstyle{definition}

\newtheorem*{definition*}{\bf Definition}

\newtheorem{remark}{\bf Remark}
\newtheorem*{remark*}{\sc Remark}

\newtheorem*{example*}{\bf Example}

\newcommand{\loc}{{\rm loc}}

\newcommand{\sprt}{{\rm sprt\,}}

\numberwithin{equation}{section}

\begin{document}

\fontsize{10.4pt}{4.5mm}\selectfont

\title[]{Non-local parabolic equations with singular (Morrey) time-inhomogeneous  drift}

\author{D.\,Kinzebulatov}

\begin{abstract}
We obtain Sobolev regularity estimates for solutions of non-local parabolic equations with locally unbounded drift satisfying some minimal assumptions. 
These results yield Krylov bound for the corresponding Feller stable process as well as some a priori regularity estimates on solutions of McKean-Vlasov equations. A key element of our arguments is a parabolic operator norm inequality that we prove using some ideas of Adams and Krylov.
\end{abstract}

\email{damir.kinzebulatov@mat.ulaval.ca}

\address{Universit\'{e} Laval, D\'{e}partement de math\'{e}matiques et de statistique, Qu\'{e}bec, QC, Canada}



\thanks{The research of the author is supported by  NSERC grant (RGPIN-2024-04236)}

\maketitle

\section{Introduction}

The present paper develops Sobolev regularity theory of non-local parabolic equation 
\begin{equation}
\label{eq0}
\partial_t u + (- \Delta)^{\frac{\alpha}{2}}u + b \cdot \nabla u=f, \qquad 1<\alpha <2,
\end{equation}
and of terminal-value problem 
\begin{equation}
\label{eq_t}
\left\{
\begin{array}{l}
-\partial_t w + (- \Delta)^{\frac{\alpha}{2}}w + b\cdot \nabla w=|b|, \quad t<r,\\[2mm]
w(r,\cdot)=0
\end{array}
\right.
\end{equation}
with locally unbounded time-inhomogeneous drift $b:\mathbb R^{d+1} \rightarrow \mathbb R^d$ satisfying some minimal assumptions (Theorems \ref{thm1}, \ref{thm2} and Corollaries \ref{krylov_thm}, \ref{cor_mv}). These assumptions allow $b$ to have critical-order singularities in time and in spatial variables.  Throughout the paper, dimension $d \geq 2$.

There is rich literature on non-local parabolic equations with irregular drift.
A significant portion of this literature is aimed at providing tools for constructing and studying the corresponding to $(-\Delta)^{\frac{\alpha}{2}} + b\cdot \nabla$ stochastic process, called $\alpha$-stable process with drift $b$. The latter is expected to solve, in one sense or another, stochastic differential equation (SDE)
\begin{equation}
\label{sde}
X_t=x-\int_0^t b(s,X_s)ds + Z_t, \quad t \geq 0, \quad \text{$Z_t$ is isotropic $\alpha$-stable process}.
\end{equation}
The properties of the process include various types of uniqueness for SDE \eqref{sde}, bounds on the transition density, existence of the mean field limit.
We focus on the Sobolev regularity of solutions of equations \eqref{eq0}, \eqref{eq_t} and on the well-posedness of SDE \eqref{sde}, and mention the following works that deal, as the present paper does, with \textit{general} drifts (so, there are no assumptions on divergence ${\rm div\,}b$):

\smallskip

1)~Drifts $b:\mathbb R^d \rightarrow \mathbb R^d$ in Kato class \cite{CW, KSo} or in a larger class of weakly form-bounded drifts \cite{KM}.

2)~Distributional drifts $b$ in Besov spaces, see \cite{CM} and references therein. 

\smallskip

In the latter case 2), one additionally faces quite non-trivial problem of  defining  product $b \cdot \nabla u$ and giving meaning to SDE \eqref{sde}, that is, in what sense one evaluates distribution $b$ along trajectory $t \mapsto (t,X_t)$. Similar difficulties arise for measure-valued drifts in Kato class.

These classes of irregular drifts  have non-empty intersection, containing, in particular, vector fields $b:\mathbb R^d \rightarrow \mathbb R^d$ such that
\begin{equation}
\label{Lp_cond}
|b| \in L^{p}(\mathbb R^d) + L^\infty(\mathbb R^d) \quad \text{for some } p>\frac{d}{\alpha-1}.
\end{equation}
Up to replacing the strict inequality with the equality, the latter is the optimal condition on the 
scale of Lebesgue spaces. The weak existence and uniqueness in law for SDE \eqref{sde} with drift $b$ satisfying \eqref{Lp_cond} was established in \cite{P,PP}.
(By the way, if we take $p=\frac{d}{\alpha-1}$ in \eqref{Lp_cond}, this will no longer be a sub-class of the Kato class, but  will still be a rather small sub-class of the class of weakly form-bounded drifts of \cite{KM}.)

In the case $\alpha=2$, there are many more results on parabolic equations with general singular drifts and the corresponding SDEs driven by Brownian motion, both in non-distributional and distributional settings. See, in particular, \cite{FIR, HZ, Ki_Morrey, KiS_sharp, Kr4, Kr5, RZ, RZ2, YZ, ZZh}.

There is an even richer literature on parabolic equations with divergence-free drift $b$ (or, more generally, ${\rm div\,}b$ is of proper sign so that there are no strong singularities of the drift that push the process towards a point). We refer, in particular, to \cite{CoM} where the authors prove, among other results, H\"{o}lder continuity of solutions for bounded in time divergence-free drifts  that belong to a Morrey space in the spatial variables that is similar to the one in the present paper.

The divergence-free condition on the drift is, on the one hand, dictated by applications in hydrodynamics. On the other hand, it opens up a way for the arguments that are not available for general drifts.

Let us also mention the terminology coming from the study of quasi-geostrophic equation, where ``critical'' means that in \eqref{eq0} one has  $\alpha=1$, super-critical means $0<\alpha<1$ and sub-critical $1<\alpha<2$. See, in particular, \cite{CV, CoM, S, XZ, Zh} (although the drifts in these papers are still critical-order in the sense of scaling). In the critical or super-critical cases, in absence of any assumptions on ${\rm div}\,b$, one needs e.g.\,appropriate H\"{o}lder continuity of drift $b$ in order to have H\"{o}lder continuous solutions \cite{S}. In fact, if $0<\alpha \leq 1$ and $b$ is not H\"{o}lder continuous enough even at a single point, then the heat kernel of $(-\Delta)^{\frac{\alpha}{2}} + b \cdot \nabla$ can vanish in the second variable \cite{KMS}. In the case $1<\alpha<2$ one observes the same effect for locally unbounded repulsing drift $b(x)=-\kappa |x|^{-\alpha}x$ \cite{KiS_RIMS}.
This drift satisfies the assumptions of our Theorems \ref{thm1}, \ref{thm2}.

Locally unbounded and distributional drifts are ubiquitous in many physical models. Some of these models require one to deal with time-inhomogeneous singular drifts, which can be substantially more difficult to handle than time-homogeneous drifts. One obvious reason for this is that one can no longer apply elliptic techniques such as deep results of semigroup theory (e.g.\,Trotter's theorem), Dirichlet forms or Stroock-Varopoulos inequalities. This is a difficulty that we also face in the present work. 

A function $v \in L^{1+\varepsilon}_{\loc}(\mathbb R^{d+1})$, $\varepsilon>0$, is said to be in  Morrey class $E_{1+\varepsilon}$ if 
$$
\|v\|_{E_{1+\varepsilon}}=\sup_{\rho>0,(t,x) \in \mathbb R^{d+1}}\rho\biggl(\frac{1}{|C_\rho(t,x)|}\int_{C_\rho(t,x)}|v(s,y)|^{1+\varepsilon}ds dy \biggr)^{\frac{1}{1+\varepsilon}}<\infty,
$$
where $C_\rho(t,x):=\{(s,y) \in \mathbb R^{d+1} \mid t \leq s \leq t+\rho^{\alpha}, \,|y-x| \leq \rho\}$ is the parabolic cylinder of radius $\rho>0$ centered at $(t,x)$. 
Let us fix an unbounded (or ``singular'') part $b_{\mathfrak s}:=b - b_{\mathfrak b}$ of $b$, where $b_{\mathfrak b}$  is in $L^\infty(\mathbb R^{d+1})$. In Theorems \ref{thm1} and \ref{thm2}, we impose the following

\medskip

\noindent\textbf{Assumption on $b$:}
\begin{equation}
\label{morrey}
\||b_{\mathfrak s}|^{\frac{1}{\alpha-1}}\|_{E_{1+\varepsilon}} \text{ is sufficiently small}.
\end{equation}

\medskip

Note also that the parabolic scaling applied to equation \eqref{eq0} leaves Morrey norm \eqref{morrey} invariant, so we are in the so-called critical setting.
Parameter $\varepsilon>0$ is fixed sufficiently small. In fact, the smaller $\varepsilon$ is, the larger the class of admissible singularities of $b$, but, on the other hand, the more restrictive is our condition on the value of norm \eqref{morrey}.

Our class of vector fields contains e.g.\,sums of vector fields
\begin{equation}
\label{ex}
\tag{$\text{E.1}$}
|b(x)| = \kappa|x|^{-\alpha+1} \quad \text{ or } \quad |b(t)| = \kappa(t-t_0)^{-\frac{\alpha-1}{\alpha}}, \quad t_0 \in \mathbb R
\end{equation}
with small but strictly positive $\kappa$.
For these $b$, Morrey norm \eqref{morrey}  is proportional to $\kappa^{\frac{1}{\alpha-1}}$ (with the coefficient of proportionality that depends only of $d$, $\alpha$ and $\varepsilon$, and does not depend on how we define the unbounded part $b_{\mathfrak s}=b - b_{\mathfrak b}$ by fixing some bounded $b_{\mathfrak b}$).

Generally speaking, the question of admissible singularities of drift $b$ has two dimensions:

\begin{enumerate}

\item[--]
The \textit{order} of singularities of $b$ in the sense of parabolic scaling. 

\medskip

For example, the model attracting drift $$b(x)=\kappa |x|^{-\beta}x$$ is of \textit{critical order} if $\beta=\alpha$, meaning that  zooming in or zooming out does not affect the power $\beta=\alpha$ and neither decreases nor increases $\kappa$.

\medskip

\item[--]The \textit{magnitude of the critical-order singularities} of $b$, i.e.\,factor $\kappa$ in $b(x)=\kappa |x|^{-\alpha}x$.
\end{enumerate}

\begin{definition*}
We say that the singularities of $b$ are critical if the regularity of solution $u$ of equation \eqref{eq0}  actually depends on the magnitude of (critical-order) singularities of $b$. 
\end{definition*}

This is our case: adjusting scalar multiple $\kappa$ in model drifts \eqref{ex} (which changes Morrey norm \eqref{morrey} of their unbounded part) in general affects the Sobolev regularity properties of solution $u$ of \eqref{eq0}.  To the best of our knowledge, in the non-local setting, \cite{KM} and the present paper are so far the only works that attain critical singularities of the drift in the sense defined above. The present paper is an extension of the PDE part of \cite{KM} to time-inhomogeneous drifts that can have critical-order singularities in time.

Morrey norm \eqref{morrey} measures the magnitude of the critical singularities of $b$, if there are any.

Let us add that if in the attracting drift $b(x)=\kappa |x|^{-\alpha}x$ the magnitude of the singularity $\kappa$ exceeds a certain positive threshold $\kappa_\ast$, then one runs into counterexamples to solvability of SDEs and similar problems for the parabolic equations (so, the requirement of smallness of Morrey norm \eqref{morrey} is necessary). Informally, in the struggle between the stable process or the Brownian motion and the drift the latter starts to have an upper hand. This transition of phase-like effect is better understood in the local case $\alpha=2$, see \cite{FT,Ki_multi,KiS_sharp,T}.

The vector fields in \eqref{ex} are strictly more singular than the vector fields in the optimal Lebesgue class
\begin{equation}
\label{lebesgue}
\tag{$\text{E.2}$}
|b| \in L^{p}(\mathbb R,L^q(\mathbb R^d)), \qquad \frac{d}{\alpha-1} \leq p \leq \infty,\;\;\frac{d}{p}+\frac{\alpha}{q} \leq \alpha-1, 
\end{equation}
which is also a subclass of \eqref{morrey}. The unbounded part of $b$ satisfying \eqref{lebesgue} can be chosen to have arbitrarily small  Lebesgue and hence Morrey norm, and so the Sobolev regularity theory of \eqref{eq1} does not sense the scalar multiple in front of the vector field. In this sense, \eqref{lebesgue} is a sub-critical class of drifts even if $\frac{d}{p}+\frac{\alpha}{q} = \alpha-1$.
Examples \eqref{ex} and \eqref{lebesgue} do not exhaust condition \eqref{morrey}. Importantly, \eqref{morrey} allows to handle drifts with strong hypersurface singularities, such as
\begin{equation}
\label{ex_hyp}
\tag{$\text{E.3}$}
|b(x)|=|x_1|^{-s}, \quad x=(x_1,\dots,x_d),\quad 0<s<\frac{\alpha-1}{1+\varepsilon}.
\end{equation}

\medskip

Theorems \ref{thm1}, \ref{thm2} are proved using the following result:

\medskip

\begin{enumerate}
\item[]
\textit{Let
$
|b|^{\frac{1}{\alpha-1}} \in E_{1+\varepsilon}.
$
Then for every $1<p<\infty$ there exists constant $c_{d,p,\varepsilon,\alpha}$ such that}
\begin{equation}
\label{ak}
\||b|^{\frac{1}{p}}( \pm \partial_t  + (- \Delta)^{\frac{\alpha}{2}} )^{(-1+\frac{1}{\alpha})\frac{1}{p}}\|_{L^p(\mathbb R^{d+1}) \rightarrow L^p(\mathbb R^{d+1})} \leq c_{d,p,\varepsilon,\alpha}\||b|^{\frac{1}{\alpha-1}}\|^{\frac{\alpha-1}{p}}_{E_{1+\varepsilon}}.
\end{equation}
\end{enumerate}

\medskip

(See detailed statement in Proposition \ref{prop1} and Corollary \ref{cor_prop1}.)

\medskip

The proof of \eqref{ak} is a modification of Krylov \cite[Theorem 4.1]{Kr3} whose  result, in turn, is a partial extension of an earlier elliptic result of Adams \cite[Theorem 8.3]{A}. Parabolic estimate \eqref{ak} is new. However, given how much its proof depends on the ideas of these two authors, it is fair to call it Adams-Krylov inequality. See also \cite{Kr6,Kr7}.

Inequality \eqref{ak} allows us to establish representations \eqref{u_repr0} and \eqref{u_repr} for the Duhamel series for solution $u$ of parabolic equation \eqref{eq0}, from which our Sobolev regularity results follow rather quickly. The problem of giving proper meaning to the formal Duhamel series for \eqref{eq0} for irregular drift was studied in a multitude of papers, including many papers cited above. 
In particular, in the elliptic setting, similar to \eqref{u_repr0} and \eqref{u_repr} 
fractional representation for the Neumann series were used in \cite{Ki_super, Ki_revisited, LS}, both for non-self-adjoint and self-adjoint operators.
Our point is that \eqref{u_repr0} and \eqref{u_repr} is the ``right guess'' for the Duhamel series for \eqref{eq1} which allows one to handle drifts $b$ that can have critical-order singularities both in time and in spatial variables.

In a paper with K.R.\,Madou \cite{KM}, we handled time-homogeneous weakly form-bounded drifts that can be more singular in the spatial variables than the ones in Theorems \ref{thm1}, \ref{thm2}. However, the analytic part of \cite{KM} depends on elliptic techniques, more precisely, on the Stroock-Varopoulos inequalities for  symmetric Markov generators (symmetry is important there; in our setting we have Markov generators $\pm \partial_t+(-\Delta)^{\frac{\alpha}{2}}$, but they are, evidently, not symmetric). 
It is not yet clear how to extend the approach of \cite{KM} to time-inhomogeneous drifts $b$. In the present paper, we follow another path and use inequality \eqref{ak} instead of the Stroock-Varopoulos inequalities. Compared to \cite{KM}, we no longer include the Kato class of of time-homogeneous drifts, but we gain instead in the strong singularities of $b$ in time.

The present work is a follow-up to \cite{Ki_Morrey} that dealt with time-inhomogeneous Morrey class drifts in SDEs driven by Brownian motion. For these SDEs, we proved in \cite{Ki_Morrey} weak existence and conditional weak uniqueness for every initial point. That said, in the local case $\alpha=2$, there are alternatives \cite{Ki_multi,KiS_sharp} to the argument of \cite{Ki_Morrey}. 
Although these alternatives handle (for $d \geq 3$ and for the backward Kolmogorov equation) a somewhat smaller class of time-inhomogeneous form-bounded vector fields, they have some important advantages, such as much better control on the magnitude of the singularities of $b$. In the non-local case $1<\alpha<2$, the range of available techniques for critical singular drifts is somewhat more limited, at least at the moment.

Regarding weak and strong well-posedness of SDEs driven by Brownian motion with discontinuous (BMO) diffusion coefficients and drifts in somewhat smaller Morrey classes (roughly, with $E_{1+\varepsilon}$ replaced by $E_{\frac{d}{2}+\varepsilon}$), see recent papers \cite{Kr4,Kr5}.

Let us also add that having dimension-independent conditions on drift $b$ is quite useful in the study of multi-particle systems with singular interactions, see \cite{Ki_multi}. That said, although space $E_{1+\varepsilon}$ is dimension-invariant, the method of the present paper imposes dimension-dependent conditions on the value of $\|b\|_{E_{1+\varepsilon}}$ (we overcome this in \cite{Ki_multi} in the case $\alpha=2$ using De Giorgi's method).

\medskip

We obtained Theorems \ref{thm1}, \ref{thm2} with the following two applications in mind. 

\medskip

\textbf{a}) Let $b$ have compact support.
Theorem \ref{thm2} gives us probability measures $\{\mathbb P_x\}_{x \in \mathbb R^d}$ on the space of c\`{a}dl\`{a}g trajectories such that, for every $x \in \mathbb R^d$, we have:

1. $\mathbb P_x[\omega_0=x]=1$ ($\omega_t$ is the canonical process),

2.~\begin{equation}
\label{b_}
\mathbb E_{\mathbb P_x}\int_0^t |b(s,\omega_s)|ds<\infty
\end{equation}
(so the process does not spend too much time around the singularities of $b$), and 

3.$$
\mathbb E_{\mathbb P_x}[g(\omega_r)]=P^{0,r}g(x), \quad r \geq 0, \quad \text{ for all } g \in C^\infty_c(\mathbb R^d),
$$
where $\{P^{t,r}\}_{t \leq r}$ is a backward Feller evolution family for $(-\Delta)^{\frac{\alpha}{2}} + b\cdot \nabla$, so that function $v(t,x)=P^{t,r}g(x)$ solves for $t<r$ (in a weak sense of in the sense of approximations) terminal-value problem 
$$
\left\{
\begin{array}{l}
(-\partial_t  + (- \Delta)^{\frac{\alpha}{2}} + b(t,x) \cdot \nabla)v=0, \\[2mm]
v(r,\cdot)=g(\cdot).
\end{array}
\right.
$$

A crucial result needed to conclude that $\mathbb P_x$ is weak solution of SDE \eqref{sde}, i.e.\,
\begin{equation}
\label{Zt}
Z_t:=\omega_t-x-\int_0^t b(s,\omega_s)ds
\end{equation}
is an $\alpha$-stable process with respect to $\mathbb P_x$, is the following Krylov-type bound, which we obtain in Section \ref{krylov_sect} rather easily by means of Theorem \ref{thm1}: 

\begin{enumerate}
\item[]
\textit{If $\||b_{\mathfrak s}|^{\frac{1}{\alpha-1}}\|_{E_{1+\varepsilon}}$ is sufficiently small, then, for every function $F$ on $\mathbb R^d$ (here, for simplicity, also of compact support) such that  $$\||F|^{\frac{1}{\alpha-1}}|\|_{E_{1+\varepsilon}}<\infty$$ 
we have
\begin{align}
\label{krylov}
\mathbb E_{\mathbb P_x} \int_0^t |F(s,\omega_s)| ds & \leq C \|F\|_{L^1([0,t] \times \mathbb R^d)}^{\frac{1}{p}}, \quad 0 \leq t \leq 1,
\end{align}
with constant $C$ that depends only on $\||b_{\mathfrak s}|^{\frac{1}{\alpha-1}}|\|_{E_{1+\varepsilon}}$, $\||F|^{\frac{1}{\alpha-1}}|\|_{E_{1+\varepsilon}}$, $\alpha$, $d$, $\varepsilon$ and $p$.}
\end{enumerate}

\smallskip

See Corollary \ref{krylov_thm} for detailed statement.

\medskip

In particular, selecting $F=b$, we obtain \eqref{b_}. After taking $F:=b_n-b_m$, where $\{b_n\}$ are mollified cutoffs of $b$ defined by \eqref{b_n}, we obtain an estimate that can be used to pass to the limit in the approximating SDEs for \eqref{sde}; this was done in \cite{KM} where we dealt with time-homogeneous drifts by means of weighted variant of \eqref{krylov} and an $L^p$ variant of the approach of Portenko \cite{P} and Podolynny-Portenko \cite{PP}.

In a subsequent work, we plan to extend the probabilistic part of \cite{KM} to time-inhomogeneous Morrey drifts, i.e.\,to show that, for every initial point $x \in \mathbb R^d$, $\mathbb P_x$ is indeed a weak solution of SDE \eqref{sde}. Also, conditional weak uniqueness for \eqref{sde} was not addressed in \cite{KM}, so this needs to be dealt with as well. It should be added that such uniqueness results are usually obtained via strong gradient bounds on solutions of the backward Kolmogorov equation, i.e.\,bounds of the type obtained in Theorem \ref{thm1} (see e.g.\,\cite{CW,Ki_Morrey,RZ}).

\bigskip

\textbf{b}) The evolution of the law of a typical particle in a large ensemble of interacting particles is described by McKean-Vlasov equation 
\begin{equation}
\label{mv_eq}
\partial_t \rho + (-\Delta)^{\frac{\alpha}{2}}\rho - {\rm div\,}\big[ (b \ast \rho) \rho \big]=0, \quad  \rho(0,\cdot)=\rho_0(\cdot) \geq 0,
\end{equation}
where $h$ has integral $1$ over $\mathbb R^d$, the convolution $\ast$ is in the spatial variables, vector field $b:\mathbb R^d \rightarrow \mathbb R^d$ describes pairwise interactions between the particles, such as attraction, repulsion or more general interactions. More precisely, one expects to obtain $\rho(t,\cdot)$ as the limit $N \rightarrow \infty$ of the law of $X_t^1$, where
\begin{equation}
\label{particle_syst}
X_t^i=X_0^i - \frac{1}{N}\sum_{j=1, j \neq i}^N \int_0^t b(X^i_s-X^j_s)ds + Z_t^i, \quad t \geq 0, \quad i=1,\dots,N,
\end{equation}
is an SDE in $\mathbb R^{Nd}$ that describes the dynamics of $N$ interacting particles,
 $Z_t^i$ are independent isotropic $\alpha$-stable processes, provided that the initial conditions $X_0^i$ satisfy the exchangeability hypothesis that allows one to consider the particles as indistinguishable (e.g.\,$X_0^i$ are independent with density $h$). If \eqref{particle_syst} indeed gives rise (in some precise sense) to \eqref{mv_eq}, then one says that the propagation of chaos holds. The propagation of chaos comes in different variants corresponding to different types of convergence; we will not go into further details in this short discussion, referring instead to \cite{BJW, CJM, HRZ, To}. Let us only add that the attracting interaction kernel $b(x)=\kappa|x|^{-\alpha+1}$ from our Example \ref{ex} is particularly relevant here: the corresponding McKean-Vlasov equation is the famous Keller-Segel model of chemotaxis \cite{FT,T}.

The regularity properties of solution of McKean-Vlasov equation are of independent interest. Our goal here is to demonstrate how Theorem \ref{thm1} can be used to obtain a priori regularity estimates for equation \eqref{mv_eq}. ``A priori'' means that $b$ and $h$ are additionally assumed to be bounded and smooth, but the  constants in the regularity estimates
do not depend on the boundedness or smoothness of $b$ or $h$.

Let us take here for simplicity $b_{\mathfrak s}=b$, so we will impose condition \eqref{morrey} on $b$ itself.

 Since $b$ is time-homogeneous (but, by the way, the non-linear drift $b \ast \rho$ in \eqref{mv_eq} is time-inhomogeneous), it will be convenient to work with the elliptic Morrey norm  
\begin{equation}
\label{morrey_elliptic}
\||b|^{\frac{1}{\alpha-1}}\|_{M_{1+\varepsilon}}:=\sup_{\rho>0,x \in \mathbb R^{d}}\rho\biggl(\frac{1}{|B_\rho(x)|}\int_{B_\rho(x)}|b|^{\frac{1+\varepsilon}{\alpha-1}}dy \biggr)^{\frac{1}{1+\varepsilon}},
\end{equation}
where $B_\rho(x)$ is the ball of radius $\rho$ centered at point $x$.
Of course, for such $b$
$$
\||b|^{\frac{1}{\alpha-1}}\|_{M_{1+\varepsilon}}=\||b|^{\frac{1}{\alpha-1}}\|_{E_{1+\varepsilon}}.
$$
Fix $\lambda>0$. Let $\mathbb{W}_\alpha^{\gamma,p}(\mathbb R^{d+1})$ denote the fractional parabolic Bessel potential space:
$$
\mathbb{W}_\alpha^{\gamma,p}(\mathbb R^{d+1}):=\big(\lambda+\partial_t + (-\Delta)^{\frac{\alpha}{2}}\big)^{-\frac{\gamma}{2}}L^p(\mathbb R^{d+1})
$$
with norm 
$$
\|g\|_{\mathbb{W}_\alpha^{\gamma,p}}:=\|\big(\lambda + \partial_t + (-\Delta)^{\frac{\alpha}{2}}\big)^{\frac{\gamma}{2}}g\|_p.
$$
Let $e_\lambda(t):=e^{-\lambda t}$.
Using Theorem \ref{thm1}, we establish in Section \ref{mv_sect} the following global in time regularity result for $\rho$:

\begin{enumerate}
\item[]
\textit{Let $1<p<\infty$. If $\||b|^{\frac{1}{\alpha-1}}\|_{M_{1+\varepsilon}}$ is sufficiently small, then
\begin{equation}
\label{mv_emb0}
\|e_\lambda\rho\|_{\mathbb W^{\frac{2}{p}(1-\frac{1}{\alpha}),p}_\alpha(\mathbb R^{d+1})} \leq C \big\|\big(\lambda+\partial_t+(-\Delta)^{\frac{\alpha}{2}}\big)^{-\frac{1}{p'}-\frac{1}{\alpha p}}\delta_0\,h\big\|_{L^p(\mathbb R^{d+1})},
\end{equation}
where $\delta_0$ is the delta-function in the time variable concentrated at $t=0$ (see \eqref{delta_def}), constant $C$ is independent of boundedness or smoothness of $b$ and $h$.
In particular,
\begin{equation}
\label{mv_emb}
\|e_\lambda\rho\|_{\mathbb W^{\frac{2}{p}(1-\frac{1}{\alpha}),p}_\alpha(\mathbb R^{d+1})} \leq C_1 \|h\|_{L^q(\mathbb R^d)} \quad \text{ for all }\frac{d}{d+1}\,p<q \leq p.
\end{equation}
}
\end{enumerate}

See Corollary \ref{cor_mv} for detailed statement.

\medskip

The following parabolic Sobolev embedding property can be used in \eqref{mv_emb0}, \eqref{mv_emb}: for all
 $1 \leq p \leq q \leq \infty$ such that
\begin{equation}
\label{sob_emb0}
\frac{\gamma}{2}>\biggl(1+\frac{d}{\alpha} \biggr)\bigg(\frac{1}{p}-\frac{1}{q}\bigg)
\end{equation}
one has
\begin{equation}
\label{sob_emb}
\big(\lambda+\partial_t + (-\Delta)^{\frac{\alpha}{2}}\big)^{-\frac{\gamma}{2}} \in \mathcal B(L^p,L^q), \qquad \text{ so }\mathbb{W}_\alpha^{\gamma,p} \subset L^q(\mathbb R^{d+1}),
\end{equation}
see \cite{AS}, see also a detailed discussion of parabolic Morrey-Sobolev embeddings in \cite{AX}.

\medskip

Estimates \eqref{mv_emb0}, \eqref{mv_emb} are ``naive'' in the sense that they are obtained using an argument that ignores the regularization by convolution. Namely, \eqref{mv_emb0}, \eqref{mv_emb} follow from inequality
\begin{equation}
\label{mv_morrey}
\||b \ast e_\lambda \rho|^{\frac{1}{\alpha-1}}\|_{E_{1+\varepsilon}} \leq \||b|^{\frac{1}{\alpha-1}}\|_{M_{1+\varepsilon}}, \quad e_\lambda(t)=e^{\lambda t},
\end{equation}
and the dual variant of Theorem \ref{thm1} for the forward Kolmogorov equation. In other words, we act as if $\rho$ is the delta-function at all times. So, one can hope that \eqref{mv_emb0}, \eqref{mv_emb} are close to the optimal regularity estimates only around initial time $t=0$, and for $p$ close to $1$ (so that the initial denisty $h$ can be chosen to be (relatively) close to the delta-function, before the regularization by convolution kicks in). Still, taking into account \eqref{mv_emb0}, \eqref{mv_emb}, one can establish at the next step a better Morrey regularity of $b \ast \rho$ than the one provided \eqref{mv_morrey}, and hence improve a priori estimates  \eqref{mv_emb0}, \eqref{mv_emb}. In other words, one can iterate these regularity results in order to approach the optimal ones; we plan to address this elsewhere.

\subsection*{Notations}

\begin{enumerate}[label=(\alph*)]

\item Put
$$
\langle h \rangle :=\int_{\mathbb R^{d+1}} h dz, \quad \langle h,g\rangle:=\langle hg\rangle
$$
All functions considered in this paper are real-valued.

\item 
Denote by $\mathcal B(X,Y)$ the space of bounded linear operators between Banach spaces $X \rightarrow Y$, endowed with the operator norm $\|\cdot\|_{X \rightarrow Y}$, set $\mathcal B(X):=\mathcal B(X,X)$.

\item We write $T=s\mbox{-} X \mbox{-}\lim_n T_n$ for $T$, $T_n \in \mathcal B(X)$ if $Tf=\lim_n T_nf$ in $X$ for every $f \in X$. 

\item 
Put $$\|h\|_p \equiv \|h\|_{L^p(\mathbb R^{d+1})}:=\langle |h|^p\rangle^{\frac{1}{p}}.$$ 
Denote $\|\cdot\|_{p \rightarrow q}:=\|\cdot\|_{L^p(\mathbb R^{d+1}) \rightarrow L^p(\mathbb R^{d+1})}$ the corresponding operator norm.

\item Let $C_\infty(\mathbb R^{d}):=\{f \in C(\mathbb R^{d}) \mid \lim_{|x| \rightarrow \infty}f(x)=0\}$ endowed with the $\sup$-norm.

\item Given a vector field $b:\mathbb R^{d+1} \rightarrow \mathbb R^d$ and $1 <p<\infty$, we put
$b^{\frac{1}{p}}:=b|b|^{-1+\frac{1}{p}}.$

\item Recall that the integral kernel $q_\gamma(t,x)$, $0<\gamma \leq 2$, of operators
\begin{align*}
(\partial_t+(-\Delta)^{\frac{\alpha}{2}})^{-\frac{\gamma}{2}}f(t,x) &=\int_{\mathbb R^{d+1}} q_\gamma(t-s,x-y)f(s,y)dyds \\[2mm]
(-\partial_t+(-\Delta)^{\frac{\alpha}{2}})^{-\frac{\gamma}{2}}f(t,x)&=\int_{\mathbb R^{d+1}} q_\gamma(s-t,x-y)f(s,y)dyds
\end{align*}
satisfies pointwise bounds
\begin{equation}
\label{ul_bounds}
c\, p_\gamma(t,x) \leq q_\gamma(t,x)  \leq C\, p_\gamma(t,x)
\end{equation}
for
$$
p_\gamma(t,x):=\mathbf{1}_{\{t>0\}} \left\{
\begin{array}{ll}
t^{\frac{\gamma}{2}}|x|^{-d-\alpha}, & |x| \geq t^{\frac{1}{\alpha}}, \\
t^{\frac{\gamma}{2}-\frac{d+\alpha}{\alpha}}, & |x| < t^{\frac{1}{\alpha}},
\end{array}
\right.
$$
with constants $0<c$, $C<\infty$ that only depend on $d$, $\alpha$ and $\gamma$, see \eqref{e}.
Also,
\begin{equation}
\label{grad_bd}
|\nabla_x q_\gamma(t,x)| \leq C_1 p_{\frac{\gamma}{2}-\frac{2}{\alpha}}(t,x),
\end{equation}
see \eqref{grad_e}.

\medskip

\end{enumerate}

\bigskip

\section{Parabolic equation on $\mathbb R^{d+1}$}We begin with a Sobolev regularity theory of  inhomogeneous parabolic equation
\begin{equation}
\label{eq1}
\lambda u + \partial_t u + (- \Delta)^{\frac{\alpha}{2}}u + b \cdot \nabla u=f \quad \text{ on } \mathbb R^{d+1}.
\end{equation}
The ``fractional'' Duhamel series representation of solution $u$, given by formula \eqref{u_repr}, is the key to the rest of the paper.

Set
\begin{equation}
\label{b_n}
b_n:=\gamma_{\varepsilon_n} \ast \mathbf{1}_n b, 
\end{equation}
where $\mathbf{1}_n$ is the indicator function of set $\{(t,x) \in \mathbb R^{d+1} \mid |t| \leq n, |x| \leq n, |b(t,x)| \leq n\}$, $\{\gamma_\epsilon\}_{\epsilon>0}$ is the Friedrichs mollifier on $\mathbb R^{d+1}$ and $\epsilon_n \downarrow 0$ sufficiently rapidly.
It should be added that the mollifier is not necessary for our arguments, but it will be convenient to work with $b_n$ that are not only bounded with compact support but are also smooth. Selecting $\varepsilon_n \downarrow 0$ sufficiently rapidly we can make $b_n$ as close to $\bar{b}_n:=\mathbf{1}_n b$ as needed in $L^s(\mathbb R^{d+1})$ for arbitrarily large $s<\infty$.

\medskip

In the next theorem we are mostly interested in the least restrictive condition on the order of admissible singularities of of $b$, i.e., we are interested in the values of parameter $q$ close to $1$.

\begin{theorem}
\label{thm1} Let $b=b_{\mathfrak{s}}+b_{\mathfrak{b}}$, where $|b_{\mathfrak b}| \in L^\infty(\mathbb R^{d+1})$ and 
$$
|b_{\mathfrak s}|^{\frac{1}{\alpha-1}} \in E_{q} \text{ for some $1<q<d+\alpha$}.
$$
The following are true:

\begin{enumerate}[label=(\roman*)]

\item
For every $1<p<\infty$, there exist constants $c_{d,\alpha,p,q}>0$ and $\lambda_{d,\alpha,p,q} \geq 0$ such that if  Morrey norm
$$\||b_{\mathfrak s}|^{\frac{1}{\alpha-1}}\|_{E_{q}} < c_{d,\alpha,p,q}$$ 
then, given any right-hand side $f \in \mathbb{W}_\alpha^{(-1+\frac{1}{\alpha})\frac{2}{p'},p}(\mathbb R^{d+1})$ of \eqref{eq1} and a sequence of functions
$\{f_n\} \subset L^\infty(\mathbb R^{d+1})$ such that $$
f_n \rightarrow f \quad \text{ in } \mathbb{W}_\alpha^{(-1+\frac{1}{\alpha})\frac{2}{p'},p}
$$
solutions $u_n$ of the approximating parabolic equations
\begin{equation*}
\lambda u_n + \partial_t u_n +(- \Delta)^{\frac{\alpha}{2}} u_n + b_n \cdot \nabla u_n=f_n, \qquad \lambda > \lambda_{d,\alpha,p,q}
\end{equation*}
converge as $n \rightarrow \infty$ in $\mathbb{W}_\alpha^{\frac{2}{\alpha}+\frac{\alpha-1}{\alpha}\frac{2}{p},p}$ to function
\begin{equation}
\label{u_repr0}
u=(\lambda+\partial_t  +(- \Delta)^{\frac{\alpha}{2}})^{-\frac{1}{\alpha}+(-1+\frac{1}{\alpha})\frac{1}{p}}(1+Q_pR_p)^{-1} (\lambda+\partial_t  +(- \Delta)^{\frac{\alpha}{2}})^{(-1+\frac{1}{\alpha})\frac{1}{p'}} f
\end{equation}
where
\begin{equation}
R_p:=b^{\frac{1}{p}}\cdot \nabla (\lambda+\partial_t  +(- \Delta)^{\frac{\alpha}{2}} )^{-\frac{1}{\alpha}+(-1+\frac{1}{\alpha})\frac{1}{p}}, 
\end{equation}
\begin{equation}
\label{Q_def}
Q_p:=(\lambda+\partial_t  +(- \Delta)^{\frac{\alpha}{2}})^{(-1+\frac{1}{\alpha})\frac{1}{p'}}|b|^{\frac{1}{p'}},
\end{equation}
are bounded operators on $L^p(\mathbb R^{d+1})$, and
\begin{equation}
\label{T_est}
\|R_p\|_{p \rightarrow p}\cdot\|Q_p\|_{p \rightarrow p} <1
\end{equation}
so the geometric series in \eqref{u_repr0} converges in $L^p(\mathbb R^{d+1})$.
In particular, it follows that $$u \in \mathbb{W}_\alpha^{\frac{2}{\alpha}+\frac{\alpha-1}{\alpha}\frac{2}{p},p}.$$
Thus, if $p>d+1$, then, by the Sobolev embedding \eqref{sob_emb0}, \eqref{sob_emb}, $u$ is bounded and continuous and the convergence of $u_n$ to $u$ is uniform on $\mathbb R^{d+1}$.

\smallskip

\item Another representation for $u$:
\begin{align}
& u =  (\lambda+\partial_t  +(- \Delta)^{\frac{\alpha}{2}})^{-1}f \notag \\
& - (\lambda+\partial_t  +(- \Delta)^{\frac{\alpha}{2}})^{-\frac{1}{\alpha}+(-1+\frac{1}{\alpha})\frac{1}{p}}Q_p (1+R_pQ_p)^{-1}R_p (\lambda+\partial_t  +(- \Delta)^{\frac{\alpha}{2}})^{(-1+\frac{1}{\alpha})\frac{1}{p'}} f. \label{u_repr}
\end{align}

\smallskip

\item The function $u$ defined by \eqref{u_repr0} or \eqref{u_repr} is the unique weak solution to 
 parabolic equation \eqref{eq1} with the right-hand side $f \in \mathbb W_\alpha^{-1+\frac{1}{\alpha},2}$, i.e.
\begin{align}
\langle (\lambda+\partial_t&+(-\Delta)^{\frac{\alpha}{2}})^{\frac{1}{2}+\frac{1}{2\alpha}}u, (\lambda + \partial_t+(-\Delta)^{\frac{\alpha}{2}})^{\frac{1}{2}+\frac{1}{2\alpha}}\eta \rangle \notag \\[2mm]
&+ \langle R_2 (\lambda+\partial_t+(-\Delta)^{\frac{\alpha}{2}})^{\frac{1}{2}+\frac{1}{2\alpha}}u,Q_2^\ast(\lambda+\partial_t+(-\Delta)^{\frac{\alpha}{2}})^{\frac{1}{2}+\frac{1}{2\alpha}} \eta \rangle \notag \\[2mm]
&=\langle (\lambda + \partial_t+(-\Delta)^{\frac{\alpha}{2}})^{-\frac{1}{2}+\frac{1}{2\alpha}}f,(\lambda+\partial_t+(-\Delta)^{\frac{\alpha}{2}})^{\frac{1}{2}+\frac{1}{2\alpha}}\eta \rangle, \label{weak_sol_def}
\end{align}
for all $\eta \in C_c^\infty(\mathbb R^{d+1})$.
\end{enumerate}
\end{theorem}

In the proof of Theorem \ref{thm1} we control the norms of operators $R_p$, $Q_p$ by means of the following parabolic estimate (which we call Adams-Krylov inequality).

\begin{proposition}
\label{prop1}
Let
$
|b|^{\frac{1}{\alpha-1}} \in E_{q}
$ for some $1<q<d+\alpha$.
Then for every $1<p<\infty$ there exists constant $C_{d,p,q,\alpha}$ such that
\begin{align}
\label{req_ineq}
|\langle |b|(P_{(1-\frac{1}{\alpha})\frac{2}{p}}f)^p\rangle| & \leq C_{d,p,q,\alpha}^p\||b|^{\frac{1}{\alpha-1}}\|_{E_q}^{\alpha-1}\|f\|_p^p, \\
|\langle |b|(P^\ast_{(1-\frac{1}{\alpha})\frac{2}{p}}f)^p\rangle| & \leq C_{d,p,q,\alpha}^p\||b|^{\frac{1}{\alpha-1}}\|_{E_q}^{\alpha-1}\|f\|_p^p, \notag
\end{align}
where 
\begin{align*}
P_\gamma f(t,x) & :=\int_{\mathbb R^{d+1}}p_\gamma(s,y)f(t+s,x+y)dyds, \\
P_\gamma^\ast f(t,x) & :=\int_{\mathbb R^{d+1}}p_\gamma(s,y)f(t-s,x+y)dyds
\end{align*}
on $f \in C_c(\mathbb R^{d+1})$. 
\end{proposition}

By \eqref{ul_bounds}, we have 
\begin{align*}
P_\gamma f & \geq C^{-1}(-\partial_t+(-\Delta)^{\frac{\alpha}{2}})^{-\frac{\gamma}{2}} f, \\
P_\gamma^\ast f & \geq C^{-1}(\partial_t+(-\Delta)^{\frac{\alpha}{2}})^{-\frac{\gamma}{2}} f
\end{align*}
on positive $f$. Therefore, we arrive at the following corollary.

\begin{corollary}
\label{cor_prop1}
Let $b=b_{\mathfrak s} + b_{\mathfrak b}$ with $b_{\mathfrak b} \in L^\infty(\mathbb R^{d+1})$ and
$
|b_{\mathfrak s}|^{\frac{1}{\alpha-1}} \in E_{q}
$ for some $1<q<d+\alpha$.
Then, for every $1<p<\infty$ and all $\lambda >0$,
\begin{align*}
\||b|^{\frac{1}{p}}(\lambda \pm \partial_t  + & (- \Delta)^{\frac{\alpha}{2}} )^{(-1+\frac{1}{\alpha})\frac{1}{p}}\|_{L^p(\mathbb R^{d+1}) \rightarrow L^p(\mathbb R^{d+1})} \\
&\leq c_{d,p,q,\alpha}\||b_{\mathfrak s}|^{\frac{1}{\alpha-1}}\|^{\frac{\alpha-1}{p}}_{E_q} + c \lambda^{(-1+\frac{1}{\alpha})\frac{1}{p}}\|b_{\mathfrak b}\|^{\frac{1}{p}}_{L^\infty(\mathbb R^{d+1})}.
\end{align*}
\end{corollary}

\begin{remark} If one tries to extend Theorem \ref{thm1} to more singular $b$ (say, to weakly form-bounded $b=b(x)$ from \cite{KM}), then one quickly realizes that the representations \eqref{u_repr0} and \eqref{u_repr} of solution of \eqref{eq1} are actually quite different. Indeed, the denominators of the geometric series are, respectively, 
$$
(\lambda+\partial_t  +(- \Delta)^{\frac{\alpha}{2}})^{(-1+\frac{1}{\alpha})\frac{1}{p'}} b \cdot \nabla (\lambda+\partial_t  +(- \Delta)^{\frac{\alpha}{2}} )^{-\frac{1}{\alpha}+(-1+\frac{1}{\alpha})\frac{1}{p}},
$$
\begin{equation*}
b^{\frac{1}{p}} \cdot \nabla (\lambda +\partial_t + (-\Delta)^{\frac{\alpha}{2}} ))^{-1}|b|^{\frac{1}{p'}},
\end{equation*}
where the former has meaning for measure-valued $b$. On the other hand, it is the second representation that allows to handle weakly form-bounded drifts in \cite{KM}.
\end{remark}

\begin{remark}
The definition of weak solution in (\textit{iii}) is obtained by  multiplying parabolic equation \eqref{eq1} by test function 
$$
(\lambda-\partial_t+(-\Delta)^{\frac{\alpha}{2}})^{-\frac{1}{2}+\frac{1}{2\alpha}}(\lambda+\partial_t+(-\Delta)^{\frac{\alpha}{2}})^{\frac{1}{2}+\frac{1}{2\alpha}}\eta
$$
and integrating by parts. This definition corresponds to taking $p=2$ in Theorem \ref{thm1}(\textit{i}), but it is not difficult to modify it to work in $L^p$ for all $1<p<\infty$; we deal with $p=2$ for simplicity, and also because the constraint on $\|b_{\mathfrak{s}}\|_{E_q}$ is minimal when $p=2$.

The counterpart of \eqref{weak_sol_def} for local elliptic equation $(\lambda-\Delta+b\cdot \nabla )u=f$ on $\mathbb R^d$,
is as follows:
\begin{equation}
\label{weak_e}
\langle (\lambda-\Delta)^{\frac{3}{4}}u, (\lambda -\Delta)^{\frac{3}{4}}\eta \rangle + \langle R^e_2 (\lambda-\Delta)^{\frac{3}{4}}u,(Q^e_2)^\ast(\lambda-\Delta)^{\frac{3}{4}} \eta \rangle 
=\langle f,(\lambda-\Delta)^{\frac{1}{2}}\eta \rangle,
\end{equation}
where $R^e_2=b^{\frac{1}{2}}\cdot \nabla(\lambda-\Delta)^{-\frac{3}{4}}$, $(Q_2^e)^\ast=|b|^{\frac{1}{2}}(\lambda-\Delta)^{-\frac{1}{4}}$, which corresponds to formally multiplying  the elliptic equation by test function $(\lambda-\Delta)^{\frac{1}{2}}\eta$ and integrating by parts. In other words, we work in the Hilbert triple $\mathcal W^{\frac{3}{2},2} \hookrightarrow \mathcal W^{\frac{1}{2},2} \hookrightarrow \mathcal W^{-\frac{1}{2},2}$ of elliptic Bessel potential spaces. If
\begin{equation}
\label{q_2}
\|(Q_2^e)^\ast\|_{L^2(\mathbb R^d) \rightarrow L^2(\mathbb R^d)}<1 \quad (\text{and so }\|R^e_2\|_{L^2(\mathbb R^d) \rightarrow L^2(\mathbb R^d)}<1), 
\end{equation}
then equation $(\lambda-\Delta+b\cdot \nabla )u=f$, $f \in \mathcal W^{-\frac{1}{2},2} $ has a unique weak solution in the sense of \eqref{weak_e} (by a standard argument, or see the proof of Theorem \ref{thm1}(\textit{ii})). In turn, by the Adams estimate \cite[Theorem 8.3]{A}, \eqref{q_2} holds if the elliptic Morrey norm
$$
\sup_{\rho>0, x\in \mathbb R^{d}}\rho\biggl(\frac{1}{|B_\rho|}\int_{B_\rho(x)}|b_{\mathfrak s}(y)|^{1+\varepsilon}dy \biggr)^{\frac{1}{1+\varepsilon}}
$$
of an unbounded part $b_{\mathfrak s}$ of $b$ 
is sufficiently small.

Let us emphasize that even in the local elliptic case we work in a non-standard triple of Hilbet spaces. If we were to work with a more restrictive (especially in small dimensions) condition 
$$
\sup_{\rho>0, x\in \mathbb R^{d}}\rho\biggl(\frac{1}{|B_\rho|}\int_{B_\rho(x)}|b_{\mathfrak s}(y)|^{2+\varepsilon}dy \biggr)^{\frac{1}{2+\varepsilon}} \text{ is small},
$$
then we would be able to stay within the standard Hilbert triple $W^{1,2}(\mathbb R^d) \hookrightarrow L^2(\mathbb R^d) \hookrightarrow W^{-1,2}(\mathbb R^d)$ (after rewriting the definitions of $R^e_2$ and $(Q^e_2)^\ast$ in terms of $b$ instead of $b^{\frac{1}{2}}$).

Definition \eqref{weak_sol_def} is a non-local parabolic extension of the previous construction. 
\end{remark}

\begin{remark}
Since operator $(\lambda+\partial_t  +(- \Delta)^{\frac{\alpha}{2}})^{(-1+\frac{1}{\alpha})\frac{1}{p'}} $ is bounded on $L^p(\mathbb R^{d+1})$, any function $f \in L^p(\mathbb R^{d+1})$ satisfies conditions of the theorem, and we can select $$f_n:=\mathbf{1}_{\{|f(x)| \leq n\}}f.$$ In the general case, to construct $\{f_n\}$, it suffices to mollify $f$ by $\varepsilon_n^{-1}(\varepsilon_n^{-1} + \partial_t + (-\Delta)^{\frac{\alpha}{2}})^{-1}$, $\varepsilon_n \downarrow 0$. 
\end{remark}

\begin{remark}
The definition  of operator $Q_p$ requires a comment. The meaning of \eqref{Q_def}
is
$$
Q_p g(t,x):=\int_{\mathbb R^{d+1}}e^{-\lambda (t-s)} q_{(1-\frac{1}{\alpha})\frac{1}{p'}}(t-s,x-y)|b(s,y)|^{\frac{1}{p'}}g(s,y)dsdy,
$$
i.e.\,in the definition of $Q_p$ we consider $(\lambda+\partial_t  +(- \Delta)^{\frac{\alpha}{2}})^{(-1+\frac{1}{\alpha})\frac{1}{p'}}|b|^{\frac{1}{p'}}$ as a whole. To consider the latter as the composition of two operators, we would need to know a priori that $|b|^{1/p'} g$, with an arbitrary $g \in L^p(\mathbb R^{d+1})$, is in the domain of $(\lambda+\partial_t  +(- \Delta)^{\frac{\alpha}{2}})^{(-1+\frac{1}{\alpha})/p'} \in \mathcal B(L^p(\mathbb R^{d+1}))$, but we do not. So, in this case we should define
$$
Q_p:=\bigl[\,(\lambda+\partial_t  +(- \Delta)^{\frac{\alpha}{2}})^{(-1+\frac{1}{\alpha})\frac{1}{p'}}|b|^{\frac{1}{p'}} \upharpoonright \mathcal E\,\bigr]^{\rm clos}_{p \rightarrow p},
$$
where $\mathcal E:=\cup_{\varepsilon>0}e^{-\varepsilon |b|}L^p(\mathbb R^{d+1})$ is a dense subspace of $L^p(\mathbb R^{d+1})$, and
$$
\big[A \upharpoonright D(A) \cap L^p \big]^{\rm clos}_{L^p \rightarrow L^p}
$$ denotes the closure of operator $A$ as an operator $L^p \rightarrow L^p$.
\end{remark}

\bigskip

\section{Feller evolution family}

We now construct the Feller evolution family for operator $(-\Delta)^{\frac{\alpha}{2}} + b\cdot \nabla$. Our main instrument is the representation \eqref{u_repr} of solution of non-homogeneous parabolic equation \eqref{eq1}. In Section \ref{krylov_sect}, this Feller evolution family will determine for every $x \in \mathbb R^d$ a probability measure $\mathbb P_x$ on the space of c\`{a}dl\`{a}g paths -- a candidate for $\alpha$-stable process with drift $b$ departing from $x$ at time $t=0$.

Recall that bounded operators $\{U^{t,s}\}_{s\leq t} \subset \mathcal B(C_\infty(\mathbb R^d))$ such that $$U^{t,r}U^{r,s}=U^{t,s}\; \text{ for all $s \leq r \leq t$}, \quad U^{s,s}=1$$ constitute a Feller evolution family if, in addition to being positivity preserving $L^\infty$ contractions, they possess the strong  continuity property:
$$
U^{r,s}=s\mbox{-}C_\infty(\mathbb R^d)\mbox{-}\lim_{t \downarrow r}U^{t,s} \quad \text{ for all } -\infty \leq s \leq r \leq \infty. 
$$

Reversing the direction of time, one obtains the definition of a backward Feller evolution family.

\medskip

Let $b$ satisfy the assumptions of Theorem \ref{thm1}. Given $n \geq 1$ and some fixed $r \in \mathbb R$, let $v_n$ denote the solution of initial-value problem 
\begin{equation}
\label{cauchy}
\left\{
\begin{array}{l}
(\lambda + \partial_t+(-\Delta)^{\frac{\alpha}{2}} + b_n \cdot \nabla)v_n  =0, \quad (t,x) \in ]r,\infty[ \times \mathbb R^d, \\[2mm]
 v_n(r,\cdot)=g \in C_\infty(\mathbb R^d),
\end{array}
\right.
\end{equation}
where $b_n$'s are defined by \eqref{b_n}. 
By a standard result, for every $n \geq 1$, operators $$U_n^{t,r}g:=v_n(t), \quad -\infty<r \leq t<\infty$$ constitute a Feller evolution family on $C_\infty(\mathbb R^d)$. 
We need to add factor $e^{\lambda (t-r)}$ to recover the conservation of probability:
$$
\int_{\mathbb R^d} e^{\lambda (t-r)}U_n^{t,r}(x,y)dy=1, \quad x \in \mathbb R^d. 
$$
Still, it is convenient to add term $\lambda>0$ in \eqref{cauchy} since we are going to use Theorem \ref{thm1}.

\medskip

Let $\delta_r(t)$ denote the delta-function (in time) concentrated at $t=r$. Put
\begin{equation}
\label{delta_def}
(\lambda+\partial_t+(-\Delta)^{\frac{\alpha}{2}})^{-\frac{\gamma}{2}}\delta_{r}\,g(t,x):=\mathbf{1}_{t > r}e^{-\lambda (t-r)}\int_{\mathbb R^d} q_{\gamma}(t-r,x-y) g(y)dy.
\end{equation}

\begin{theorem}
\label{thm2}
Under the assumptions of Theorem \ref{thm1}, let $\||b_{\mathfrak s}|^{\frac{1}{\alpha-1}}\|_{E_q} < c_{d,\alpha,p,q}$ for a $p>d+1$. Then the following are true:

\smallskip

{\rm(\textit{i})} The limit
$$
U^{t,r}:=s\mbox{-}C_\infty(\mathbb R^d)\mbox{-}\lim_n U_n^{t,r}  \quad \text{locally uniformly in $(r,t) \in \mathbb R^2$, $r \leq t$}
$$
exists and determines a Feller evolution family on $C_\infty(\mathbb R^d)$.

\smallskip

{\rm(\textit{ii})} For every initial function $g \in C_\infty(\mathbb R^d) \cap W^{1,p}(\mathbb R^d)$, function $v(t):=U^{t,r}g$ has representation
\begin{equation}
\label{v_repr}
v=(\lambda+\partial_t+(-\Delta)^{\frac{\alpha}{2}})^{-1}\delta_{r}g - (\lambda+\partial_t+(-\Delta)^{\frac{\alpha}{2}})^{-\frac{1}{\alpha}+(-1+\frac{1}{\alpha})\frac{1}{p}}Q_p (1+R_pQ_p)^{-1}G_p \cdot  S_p g,
\end{equation}
where
$$
G_p:=b^{\frac{1}{p}} (\lambda+\partial_t+(-\Delta)^{\frac{\alpha}{2}})^{(-1+\frac{1}{\alpha})\frac{1}{p}} \in \mathcal B(L^p(\mathbb R^{d+1}),[L^p(\mathbb R^{d+1})]^d)
$$
and
$$
S_p g:=\nabla (\lambda+\partial_t+(-\Delta)^{\frac{\alpha}{2}})^{-\frac{1}{p'}-\frac{1}{\alpha p}}\delta_{r}g, \qquad
\|S_pg\|_{L^p(\mathbb R^{d+1})} \leq C_{d,\alpha,p} \|\nabla g\|_{L^p(\mathbb R^d)}.$$
\medskip
{\rm(\textit{iii})} As a consequence of \eqref{v_repr} and the parabolic Sobolev embedding, we obtain
\begin{equation}
\label{sup_bd}
\sup_{r \leq t<\infty, x \in \mathbb R^d}|v(t,x)| \leq C\|g\|_{W^{1,p}(\mathbb R^d)}.
\end{equation}
\end{theorem}

Let us emphasize that some differentiability of $g$ in (\textit{iii}) is required since in this estimate we can go up to the initial time $t=r$, hence $\|g\|_{W^{1,p}(\mathbb R^d)}<\infty$ must provide boundedness of $g$.

\medskip

The constructed in Theorem \ref{thm2} function $v(t):=U^{t,r}g$ formally solves the initial-value problem
\begin{equation}
\label{cauchy2}
\left\{
\begin{array}{l}
(\lambda + \partial_t+(-\Delta)^{\frac{\alpha}{2}} + b \cdot \nabla)v  =0, \quad (t,x) \in ]r,\infty[ \times \mathbb R^d,\\
 v(r,\cdot)=g \in C_\infty(\mathbb R^d).
\end{array}
\right.
\end{equation}
Modifying the proof Theorem \ref{thm1}(\textit{ii}), one can show that this $v$ is in fact a weak solution to the above initial-value problem. If $b$ does not depend on time, then $U^{t,r}=U^{t-r}$ is a Feller semigroup, in which case $v$ is a strong solution of \eqref{cauchy2} with the generator of the semigroup replacing the formal operator $(-\Delta)^{\frac{\alpha}{2}} + b \cdot \nabla$.

\bigskip

\section{Proof of Proposition \ref{prop1}} In the proof of Proposition \ref{prop1} we follow Krylov \cite{Kr3}, but with a  different selection of parameters since we need to deal with non-integer powers of $|b|$.

\medskip

It suffices to carry out the proof of \eqref{req_ineq} for $b_n:=\mathbf{1}_n b$, where $\mathbf{1}_n$ is the indicator of $\{|t| \leq n, |x| \leq n, |b(t,x)| \leq n\}$ and then use the Dominated convergence theorem. Below we write for brevity $b$ instead of $b_n$.

We will need the following notations and estimates.
Set
$$
M_\beta f(t,x):=\sup_{\rho>0}\rho^\beta \frac{1}{|C_\rho(t,x)|}\int_{C_\rho(t,x)} |f(s,y)| dy ds, \quad 0 \leq \beta \leq d+\gamma,
$$
and define the maximal function $M:=M_0$.
Also, define 
$$
\hat{M} f(t,x):=\sup_{(t,x) \in C} \frac{1}{|C|}\int_{C} |f| dy ds,
$$
where the supremum is taken over all parabolic $\alpha$-cylinders $C$ containing $(t,x)$. This function is required for harmonic-analytic arguments that are transferred from the elliptic setting (i.e.\,from the cubes). The fact that our ``cubes'' $C$ are squeezed in one of the coordinates presents no problem.

\begin{lemma}
\label{lem1}
The following are true for every $\beta \in ]\frac{\alpha \gamma}{2},d+\gamma]$, for all $(t,x) \in \mathbb R^{d+1}$:

\smallskip

{\rm (\textit{i})}~For all $\rho>0$, we have
\begin{align*}
P_\gamma(\mathbf{1}_{C^c_\rho}f)(t,x)  \leq K \rho^{\frac{\alpha \gamma}{2}-\beta }M_\beta f(t,x), & \qquad K:=\frac{1}{\frac{\alpha\gamma}{2}-\beta}\bigl(-\frac{\alpha \gamma}{2}+d+\alpha\bigr), \\[2mm]
P_\gamma(\mathbf{1}_{C_\rho}f)(t,x)  \leq N \rho^{\frac{\alpha \gamma}{2}}Mf(t,x), & \qquad N:= \frac{2}{\alpha \gamma}\bigl(-\frac{\alpha\gamma}{2}+d+\alpha\bigr)
\end{align*}
for all $0 \leq f \in C_c(\mathbb R^{d+1})$.

\smallskip

{\rm (\textit{ii})}~$$
P_\gamma f(t,x) \leq C (M_\beta f(t,x))^{\frac{\alpha \gamma}{2}\frac{1}{\beta}}(M f(t,x))^{1-\frac{\alpha \gamma}{2}\frac{1}{\beta}}, \quad C:=K^{\frac{\alpha \gamma}{2}\frac{1}{\beta}}N^{1-\frac{\alpha \gamma}{2}\frac{1}{\beta}}. 
$$

\end{lemma}
\begin{proof}[Proof of Lemma]
We repeat Krylov \cite[proof of Lemma 2.2]{Kr3} with some straightforward modifications needed to accommodate integral kernel $p_\gamma(s,y):=\mathbf{1}_{\{s>0\}} \, s^{\frac{\gamma}{2}}(|y|^{-d-\alpha} \wedge s^{-\frac{d+\alpha}{\alpha}})$ instead of the Gaussian density $(4\pi s)^{-\frac{d}{2}}e^{-|y|^2/s}$ times $s^{\frac{\gamma}{2}-1}$. 

\smallskip

(\textit{i}) It suffices to carry out the proof for $(t,x)=(0,0)$. Below $s \geq 0$, $y \in \mathbb R^d$. Set $$Q:=\{(s,y) \mid |y| \geq s^{\frac{1}{\alpha}}\}, \quad Q^c:=\{(s,y) \mid |y| < s^{\frac{1}{\alpha}}\}.$$ With a slight abuse of notation, we write $p_\gamma(s,r)$ instead of $p_\gamma(s,y)$, where $r=|y|$.

The following result is obtained using integrating by parts. Let $\xi>0$ and $g:[0,\infty[ \rightarrow [0,\infty[$ be such that
$
t^{-\xi}\int_0^t g(s) ds \rightarrow 0$ as $t \uparrow \infty.
$
Then, for every $\tau \geq 0$,
\begin{equation}
\label{lem0}
\int_\tau^\infty t^{-\xi }g(t)dt \leq \beta \int_\tau^\infty t^{-\xi-1}\biggl(\int_\tau^t g(s)ds\biggr) dt.
\end{equation}

\medskip

1.~If $r \geq s^{\frac{1}{\alpha}}$, then $p_\gamma(s,r) \leq r^{\frac{\alpha \gamma}{2}-d-\alpha}$. Therefore,
\begin{align*}
P_\gamma(f\mathbf{1}_{Q \cap C_\rho^c})(0) & \leq \int_0^\infty\int_{\mathbb R^{d}} |y|^{\frac{\alpha \gamma}{2}-d-\alpha}f(s,y)\mathbf{1}_{Q \cap C_\rho^c}dy ds \\
& = \int_{\{0 \leq s \leq |y|^\alpha, |y| \geq \rho\}} |y|^{\frac{\alpha \gamma}{2}-d-\alpha}f(s,y)dy ds \\
& = \int_{\rho}^\infty r^{\frac{\alpha \gamma}{2}-d-\alpha} \int_0^{r^\alpha} \int_{\{|y|=r\}}f(s,y)d\sigma_r ds dr \\
& (\text{we apply \eqref{lem0} with $\xi \equiv C_1=-\frac{\alpha \gamma}{2}+d+\alpha$}) \\
& \leq C_1\int_{\rho}^\infty r^{\frac{\alpha \gamma}{2}-d-\alpha-1} \int_{\rho}^r \int_0^{\theta^\alpha} \int_{\{|y|=\theta\}} f(s,y)d\sigma_\theta ds d\theta dr\\
& \leq C_1\int_{\rho}^\infty r^{\frac{\alpha \gamma}{2}-d-\alpha-1} \int_{0}^r \int_0^{\theta^\alpha} \int_{\{|y|=\theta\}} f(s,y)d\sigma_\theta ds d\theta dr\\
& = C_1\int_{\rho}^\infty r^{\frac{\alpha \gamma}{2}-d-\alpha-1} I(r)dr,
\end{align*}
where  $$I(r):=\int_{C_r}f(s,y) ds dy.$$
Clearly,
$
I(r) \leq r^{d+\alpha-\beta}M_\beta f(0),
$
so
\begin{align*}
P_\gamma(f\mathbf{1}_{Q \cap C_\rho^c})(0) & \leq C_1\int_{\rho}^\infty r^{\frac{\alpha \gamma}{2}-d-\alpha-1} r^{d+\alpha-\beta} M_\beta f(0) \\
& \leq K \rho^{\frac{\alpha \gamma}{2}-\beta} M_\beta(0), \qquad K=\frac{1}{\frac{\alpha\gamma}{2}-\beta}C_1 \equiv \frac{1}{\frac{\alpha\gamma}{2}-\beta}\bigl(-\frac{\alpha \gamma}{2}+d+\alpha\bigr),
\end{align*}
where we have used the hypothesis $\beta>\frac{\alpha \gamma}{2}$.

\medskip

2.\,If $r \leq s^{\frac{1}{\alpha}}$, then $p_\gamma(s,r) = s^{\frac{\gamma}{2}-\frac{d}{\alpha}-1}$. Therefore,
\begin{align*}
P_\gamma(f\mathbf{1}_{Q^c \cap C_\rho^c})(0) & = \int_{\rho^\alpha}^\infty s^{\frac{\gamma}{2}-\frac{d}{\alpha}-1}\int_{\{|y| \leq s^{\frac{1}{\alpha}}\}}f(s,y) dy ds \\
& (\text{we apply \eqref{lem0} with $\xi \equiv C_2:=-\frac{\gamma}{2}+\frac{d}{\alpha}+1$}) \\
& \leq C_2\int_{\rho^\alpha}^\infty s^{\frac{\gamma}{2}-\frac{d}{\alpha}-2} \int_{\rho^\alpha}^s\int_{|y| \leq \vartheta^{\frac{1}{\alpha}}} f(\vartheta,y) dy d\vartheta ds \\
& \leq C_2\int_{\rho^\alpha}^\infty s^{\frac{\gamma}{2}-\frac{d}{\alpha}-2} \int_0^s\int_{|y| \leq \vartheta^{\frac{1}{\alpha}}} f(\vartheta,y) dy d\vartheta ds \\
& = C_2\int_{\rho^\alpha}^\infty s^{\frac{\gamma}{2}-\frac{d}{\alpha}-2} I(s^{\frac{1}{\alpha}})ds  = \alpha C_2 \int_\rho^\infty r^{\frac{\alpha \gamma}{2} - d-\alpha -1 }I(r)dr \\
& \leq  \alpha C_2 \int_\rho^\infty r^{\frac{\alpha \gamma}{2} - d-\alpha -1 } r^{d+\alpha-\beta} dr M_\beta f(0) \\
& = K \rho^{\frac{\alpha \gamma}{2}-\beta} M_\beta(0)
\end{align*}
(we have used $\beta>\frac{\alpha \gamma}{2}$).
Combined with Step 1, this yields the first inequality of Lemma \ref{lem1}(\textit{i}).

\medskip

Let us prove the second inequality in Lemma \ref{lem1}(\textit{i}):

\medskip

3.~We have, using integration by parts,
\begin{align*}
P_\gamma(f\mathbf{1}_{Q \cap C_\rho})(0) & \leq \int_0^\rho r^{\frac{\alpha\gamma}{2}-d-\alpha} \int_0^{r^\alpha} \int_{\{|y|=r\}} f(s,y)d\sigma_r ds dr \\
& \leq \int_0^\rho r^{\frac{\alpha\gamma}{2}-d-\alpha} \frac{\partial}{\partial r}\int_0^r \int_0^{\theta^\alpha} \int_{\{|y|=\theta\}} f(s,y) d\sigma_\theta ds d\theta dr \\
& = J_1 + C_3\int_0^{\rho} r^{\frac{\alpha\gamma}{2}-d-\alpha-1}\int_0^r \int_0^{\theta^\alpha}\int_{\{|y|=\theta\}} f(s,y) d\sigma_\theta ds d\theta dr \\
& \leq J_1 + C_3\int_0^\rho  r^{\frac{\alpha\gamma}{2}-d-\alpha-1} I(r)dr,
\end{align*}
where $C_3:=-\frac{\alpha\gamma}{2}+d+\alpha$ and
\begin{align*}
J_1 & :=\rho^{\frac{\alpha\gamma}{2}-d-\alpha}\int_0^\rho \int_0^{\theta^\alpha}\int_{\{|y|=\theta\}} f(s,y)d\sigma_\theta ds d\theta\\
&  \leq \rho^{\frac{\alpha\gamma}{2}-d-\alpha}I(\rho).
\end{align*}
Since $I(\rho) \leq \rho^{d+\alpha}Mf(0)$, we have
\begin{align*}
P_\gamma(f\mathbf{1}_{Q \cap C_\rho})(0) & \leq \rho^{\frac{\alpha\gamma}{2}-d-\alpha}\rho^{d+\alpha}Mf(0) + C_3  \int_0^\rho  r^{\frac{\alpha\gamma}{2}-d-\alpha-1} r^{d+\alpha} dr Mf(0) \\
& = \rho^{\frac{\alpha\gamma}{2}}Mf(0) + \frac{2C_3}{\alpha \gamma}\rho^{\frac{\alpha\gamma}{2}}Mf(0) \\
& = N \rho^{\frac{\alpha\gamma}{2}}Mf(0), \qquad N=1+\frac{2C_3}{\alpha \gamma} \equiv \frac{2}{\alpha \gamma}\bigl(-\frac{\alpha\gamma}{2}+d+\alpha\bigr).
\end{align*}

4.~Next,
\begin{align*}
P_\gamma(f\mathbf{1}_{Q^c \cap C_\rho})(0) & = \int_0^{\rho^\alpha} s^{\frac{\gamma}{2}-\frac{d}{\alpha}-1}\int_{\{|y| \leq s^{\frac{1}{\alpha}}\}}f(s,y) dy ds \\
& \leq J_2+ C_4 \int_0^{\rho^\alpha} s^{\frac{\gamma}{2}-\frac{d}{\alpha}-2} I(s^{\frac{1}{\alpha}}) ds \\
& = J_2 + \alpha C_4\int_0^\rho r^{\frac{\alpha \gamma}{2}-d-\alpha-1} I(r)dr,
\end{align*}
where $C_4=-\frac{\gamma}{2}+\frac{d}{\alpha}+1$ and
\begin{align*}
J_2 & :=\rho^{\frac{\alpha\gamma}{2}-d-\alpha}\int_0^{\rho^\alpha}\int_{\{|y| \leq \vartheta^{\frac{1}{\alpha}}\}} f(\vartheta,y)dy d\vartheta  \\
& \leq \rho^{\frac{\alpha\gamma}{2}-d-\alpha}I(\rho).
\end{align*}
Thus, applying again inequality $I(\rho) \leq \rho^{d+\alpha}Mf(0)$, we obtain 
$P_\gamma(f\mathbf{1}_{Q^c \cap C_\rho})(0) \leq N\rho^{\frac{\alpha\gamma}{2}}Mf(0)$. Together with the result of step 1,
this completes the proof of the second inequality in Lemma \ref{lem1}(\textit{i}). 

\medskip

(\textit{ii}) is obtained from (\textit{i}) by adding up both inequalities and minimizing in $\rho$. 
\end{proof}

Armed with Lemma \ref{lem1}(\textit{ii}), we now prove \eqref{req_ineq}. Put $u:=P_{(1-\frac{1}{\alpha})\frac{2}{p}}f$. Then 
\begin{align}
|\langle |b|(P_{(1-\frac{1}{\alpha})\frac{2}{p}}f)^p\rangle| &= |\langle |b||u|^{p-1},P_{(1-\frac{1}{\alpha})\frac{2}{p}}f\rangle| \notag \\
& \leq |\langle P^\ast_{(1-\frac{1}{\alpha})\frac{2}{p}}(|b||u|^{p-1}),f\rangle| \leq \|P^\ast_{(1-\frac{1}{\alpha})\frac{2}{p}}(|b||u|^{p-1})\|_{p'}\|f\|_p. \label{ineq2}
\end{align}
Thus, to obtain \eqref{req_ineq}, we need to bound the coefficient $\|P^\ast_{(1-\frac{1}{\alpha})\frac{1}{p}}(|b||u|^{p-1})\|_{p'}$. 
To this end, we apply pointwise estimate
\begin{align*}
P_{(1-\frac{1}{\alpha})\frac{2}{p}}^\ast(|b||u|^{p-1}) = P_{(1-\frac{1}{\alpha})\frac{2}{p}}^\ast(|b|^{\frac{1}{p}+\nu} |b|^{\frac{1}{p'}-\nu}|u|^{p-1}) \leq P_{(1-\frac{1}{\alpha})\frac{2}{p}}^\ast(|b|^{1+\nu p})^{\frac{1}{p}} (P_{(1-\frac{1}{\alpha})\frac{2}{p}}^\ast(|b|^{1-\nu p'}|u|^p))^{\frac{1}{p'}}
\end{align*} 
for a small $\nu>0$ such that $1+\nu p < q_0$ for some fixed $q_0<q$. 
Hence 
\begin{align}
\|P^\ast_{(1-\frac{1}{\alpha})\frac{1}{p}}(|b||u|^{p-1})\|_{p'}^{p'} \leq \langle |b|^{1-\nu p'}|u|^p, P_{(1-\frac{1}{\alpha})\frac{2}{p}}[P^\ast_{(1-\frac{1}{\alpha})\frac{2}{p}}(|b|^{1+\nu p})]^{\frac{1}{p-1}}\rangle. \label{PP}
\end{align}

Step 1.\,Applying  Lemma \ref{lem1}(\textit{ii}) with $\gamma:=(1-\frac{1}{\alpha})\frac{2}{p}$, $\beta:=(\alpha-1)(1+\nu p)$  (clearly, $\beta>\frac{\alpha \gamma}{2}$) or, rather, its straightforward variant for $P_\alpha^\ast$ and noting that in this case
$$
\frac{\alpha \gamma}{2}\frac{1}{\beta}=\frac{1}{p}\frac{1}{1+\nu p},
$$ 
and (this will be needed a few lines later)
\begin{align*}
M_{(\alpha-1)(1+\nu p)}|b|^{1+\nu p}(t,x) & =\sup_{\rho>0}\rho^{(\alpha-1)(1+\nu p)}\frac{1}{|C_\rho(t,x)|}\int_{C_\rho(t,x)} |b(s,y)|^{1+\nu p} dy ds \\
& \leq \sup_{\rho>0} \rho^{(\alpha-1)(1+\nu p)} \biggl(\frac{1}{|C_\rho(t,x)|}\int_{C_\rho(t,x)} |b(s,y)|^{\frac{1+\nu p}{\alpha-1}} dy ds \biggr)^{\alpha-1} \\
& = \biggl[\sup_{\rho>0} \rho \biggl(\frac{1}{|C_\rho(t,x)|}\int_{C_\rho(t,x)} |b(s,y)|^{\frac{1+\nu p}{\alpha-1}} dy ds \biggr)^{\frac{1}{1+\nu p}}\biggr]^{(\alpha-1)(1+\nu p)} \\[3mm]
& \leq \||b|^{\frac{1}{\alpha-1}}\|_{E_{1+\nu p}}^{(\alpha-1)(1+\nu p)}.
\end{align*}
we obtain 
\begin{align*}
P^\ast_{(1-\frac{1}{\alpha})\frac{2}{p}}(|b|^{1+\nu p}) & \leq N (M_{(\alpha-1)(1+\nu p)}|b|^{1+\nu p})^{\frac{1}{p}\frac{1}{1+\nu p}}(M |b|^{1+\nu p})^{1-\frac{1}{p}\frac{1}{1+\nu p}}\\
& \leq C\||b|^{\frac{1}{\alpha-1}}\|_{E_{1+\nu p}}^{\frac{\alpha-1}{p}}(\hat{M}|b|^{1+\nu p})^{1-\frac{1}{p}\frac{1}{1+\nu p}} \\
& \leq C\||b|^{\frac{1}{\alpha-1}}\|_{E_{q_0}}^{\frac{\alpha-1}{p}}(\hat{M}|b|^{1+\nu p})^{1-\frac{1}{p}\frac{1}{1+\nu p}}.
\end{align*}
At this point, let us assume that $|b|^{1+\nu p}$ is an $A_1$-weight, i.e.
\begin{equation}
\label{a1}
\hat{M}|b|^{1+\nu p} \leq C_0 |b|^{1+\nu p}
\end{equation}
for some constant $C_0$
(we will get rid of this assumption later).

\begin{remark}
Up to this point all constants were explicit. Now we picked up constant $C_0$ which, as will be explained below (Step 4), depends on the constants in some classical inequalities of harmonic analysis.
\end{remark}

Then
$$
P^\ast_{(1-\frac{1}{\alpha})\frac{2}{p}}(|b|^{1+\nu p}) \leq C_2\|b\|_{E_{q_0}}^{\frac{1}{p}}|b|^{1+\nu p - \frac{1}{p}}.
$$

Step 2.\,After applying the previous estimate in \eqref{PP}, one sees that now we need to estimate
$$
P_{(1-\frac{1}{\alpha})\frac{2}{p}} (|b|^{(1+\nu p - \frac{1}{p})\frac{1}{p-1}})=P_{(1-\frac{1}{\alpha})\frac{2}{p}} (|b|^{\frac{1}{p}+\nu p'}).
$$
Selecting $\nu$ even smaller, if needed, one may assume that $\frac{1}{p}+\nu p'<q_0$.
By  Lemma \ref{lem1}(\textit{ii}) with $\gamma:=(1-\frac{1}{\alpha})\frac{2}{p}$, $\beta:=(\alpha-1)(\frac{1}{p}+\nu p')$ (in this case, again, $\beta>\frac{\alpha \gamma}{2}$),
\begin{align*}
P_{(1-\frac{1}{\alpha})\frac{2}{p}} (|b|^{\frac{1}{p}+\nu p'}) & \leq C (M_{(\alpha-1)(\frac{1}{p}+\nu p')}|b|^{\frac{1}{p}+\nu p'})^{\frac{1}{p}\frac{1}{\frac{1}{p}+\nu p'}}(M|b|^{\frac{1}{p}+\nu p'})^{1-\frac{1}{p}\frac{1}{\frac{1}{p}+\nu p'}} \\
& \leq C\||b|^{\frac{1}{\alpha-1}}\|_{E_{\frac{1}{p}+\nu p'}}^{\frac{\alpha-1}{p}} (\hat{M}|b|^{\frac{1}{p}+\nu p'})^{1-\frac{1}{p}\frac{1}{\frac{1}{p}+\nu p'}} \\
& \leq C\||b|^{\frac{1}{\alpha-1}}\|_{E_{q_0}}^{\frac{\alpha-1}{p}} (\hat{M}|b|^{\frac{1}{p}+\nu p'})^{1-\frac{1}{p}\frac{1}{\frac{1}{p}+\nu p'}}.
\end{align*}
In addition to \eqref{a1}, let us temporarily assume that
\begin{equation}
\label{a2}
\hat{M}|b|^{\frac{1}{p}+\nu p'} \leq C_0|b|^{\frac{1}{p}+\nu p'}.
\end{equation}
Then
$$
P_{(1-\frac{1}{\alpha})\frac{2}{p}} (|b|^{\frac{1}{p}+\nu p'})  \leq C_2 \||b|^{\frac{1}{\alpha-1}}\|_{E_{q_0}}^{\frac{\alpha-1}{p}} |b|^{\nu p'}.
$$

Step 3.\,Applying the results of Steps 1 and 2  in \eqref{PP}, we obtain
$$
\|P^\ast_{(1-\frac{1}{\alpha})\frac{2}{p}}(|b||u|^{p-1})\|_{p'} \leq C_3 \||b|^{\frac{1}{\alpha-1}}\|_{E_{q_0}}^{\frac{\alpha -1}{p}} \langle |b||u|^p\rangle^{\frac{1}{p'}}.
$$
Therefore, \eqref{ineq2} yields
$
\langle |b||u|^{p}\rangle \leq C_3 \||b|^{\frac{1}{\alpha-1}}\|_{E_{q_0}}^{\frac{\alpha-1}{p}} \langle |b||u|^p\rangle^{\frac{1}{p'}}\|f\|_p,
$
hence
\begin{equation}
\label{ineq4}
\langle |b||u|^{p}\rangle^{\frac{1}{p}} \leq C_3\||b|^{\frac{1}{\alpha-1}}\|_{E_{q_0}}^{\frac{\alpha-1}{p}} \|f\|_p.
\end{equation}

Step 4.\,Now, we get rid of the assumptions \eqref{a1} and \eqref{a2} at expense of replacing $\|b\|_{E_{q_0}}$ in \eqref{ineq4} by $\|b\|_{E_q}$, where, recall, $q_0<q$. This, in turn, will give us \eqref{req_ineq}.
Fix $q_0<q_1<q$ and define $$\tilde{b}:=(\hat{M} |b|^{q_1})^{\frac{1}{q_1}}.$$ Then $\tilde{b} \geq |b|$ and $\tilde{b}$ satisfies 
\begin{equation}
\label{a3}
\hat{M}\tilde{b}^{q_0} \leq C_0\tilde{b}^{q_0},
\end{equation}
see \cite[p.158]{GR}.
Since $1+\gamma p < q_0$, $\frac{1}{p}+\gamma p'<q_0$, both inequalities \eqref{a1} and \eqref{a2} for $\tilde{b}$ follow from \eqref{a3}, and so we have
$$
\langle |b||u|^{p}\rangle^{\frac{1}{p}} \leq C_3 \|\tilde{b}^{\frac{1}{\alpha-1}}\|_{E_{q_0}}^{\frac{\alpha-1}{p}} \|f\|_p.
$$
It remains to apply inequality $\||\tilde{b}|^{\frac{1}{\alpha-1}}\|_{E_{q_0}} \leq C \||b|^{\frac{1}{\alpha-1}}\|_{E_{q}}$, which was established in \cite[proof of Theorem 4.1]{Kr3} for $\alpha=2$. In fact, \cite{Kr3} gave two proofs of this fact, one is based on the Fefferman-Stein inequality, and the other one is more direct. 
In both cases, as is mentioned in \cite{Kr3}, transferring the corresponding harmonic-analytic results from cubes to $\hat{M}$ defined on the parabolic cylinders (with $\alpha=2$)  presents no problem. (Note that our parabolic cylinders with $1<\alpha<2$ get closer to the cubes as $\alpha \downarrow 1$. So, at this last step of the proof we find ourselves between \cite{Kr3} and the classical setting of cubes in $\mathbb R^{d+1}$.) \hfill \qed

\bigskip

\section{Proof of Theorem \ref{thm1}}

First, we note that, given a representation 
$$b=b_{\mathfrak s} + b_{\mathfrak b}, \quad b_{\mathfrak s} \in E_q, \quad b_{\mathfrak b} \in L^\infty(\mathbb R^{d+1}),$$
we can write $b_n=(b_{n})_{\mathfrak s} + (b_n)_{\mathfrak b}$ in such a way that
\begin{equation}
\label{b_n_ineq}
\|(b_{n})_{\mathfrak s}\|_{E_q} \leq \|b_{\mathfrak s}\|_{E_q}, \quad (1+n^{-1})\|(b_n)_{\mathfrak b}\|_{L^\infty(\mathbb R^{d+1})} \leq \|b_{\mathfrak b}\|_{L^\infty(\mathbb R^{d+1})}
\end{equation}
for all $n \geq 1$. In fact, if instead of $b_n$ we have $\bar{b}_n=\mathbf{1}_n b$, then, as is easily seen, \eqref{b_n_ineq} holds without term $n^{-1}$ since in this case we have pointwise inequality $|b_n| \leq |b|$ on $\mathbb R^{d+1}$ for all $n$. The term $n^{-1}$ comes from computing the Morrey norm of $b_n-\bar{b}_n$, provided that $\varepsilon_n\downarrow 0$ sufficiently rapidly.

In particular, we can apply Corollary \ref{cor_prop1} to $b_n$ defined by \eqref{b_n} with constants independent of $n$.

Furthermore, selecting $\epsilon_n \downarrow 0$ sufficiently rapidly, we may assume that
\begin{equation}
\label{10_ineq}
\|b_n  - \bar{b}_n\|_{L^{10p}(\mathbb R^{d+1})} \leq \frac{1}{n^{10}},
\end{equation}
where $p$ is from Theorem \ref{thm1}, and exponent $10$ can be replaced with any other large constant.

\medskip

(\textit{i}), (\textit{ii}) The fact that $R_p(b)$ is bounded on $L^p(\mathbb R^{d+1})$ follows right away upon applying pointwise estimate \eqref{grad_bd} in order to get rid of the gradient, and then using Corollary \ref{cor_prop1}. In turn, the boundedness of $Q_p$ follows by writing, by duality,
$$\|Q_p(b)\|_{L^p(\mathbb R^{d+1}) \rightarrow L^p(\mathbb R^{d+1})}=\||b|^{\frac{1}{p'}}(\lambda - \partial_t  +(- \Delta)^{\frac{\alpha}{2}})^{(-1+\frac{1}{\alpha})\frac{1}{p'}}\|_{L^{p'}(\mathbb R^{d+1}) \rightarrow L^{p'}(\mathbb R^{d+1})}$$ and then applying Corollary \ref{cor_prop1}.
It is now clear that $\|R_p(b)Q_p(b)\|_{p \rightarrow p}<1$ provided that $\|b_{\mathfrak s}\|_{E_q}$ is sufficiently small and $\lambda$ is greater than certain $\lambda_{d,\alpha,p,q}$. 
So, the operator-valued function in assertion (\textit{ii}) (without loss of generality, we will only establish representation \eqref{u_repr} for $u$), i.e.
\begin{align*}
 \Theta_p(b):=& (\lambda+\partial_t  +(- \Delta)^{\frac{\alpha}{2}})^{-1}f \notag \\
& - (\lambda+\partial_t  +(- \Delta)^{\frac{\alpha}{2}})^{-\frac{1}{\alpha}+(-1+\frac{1}{\alpha})\frac{1}{p}}Q_p (1+R_pQ_p)^{-1}R_p (\lambda+\partial_t  +(- \Delta)^{\frac{\alpha}{2}})^{(-1+\frac{1}{\alpha})\frac{1}{p'}} 
\end{align*}
takes values in $\mathcal B(L^p(\mathbb R^{d+1}))$ for all $\lambda > \lambda_{d,\alpha,p,q}$.

Also by Corollary \ref{cor_prop1} and \eqref{b_n_ineq}, the above estimates on the norms of $R_p(b)$ and $Q_p(b)$ transfer to $R_p(b_n)$ and $Q_p(b_n)$ with constants independent of $n$. In particular,
$\|R_p(b_n)Q_p(b_n)\|_{p \rightarrow p}<1$ for all $n$, so $\Theta_p(b_n)$ is well-defined for all $\lambda > \lambda_{d,\alpha,p,q}$ as an operator-valued function taking values in $\mathcal B(L^p(\mathbb R^{d+1}))$.

By the classical theory, for every $n \geq 1$ there exists solution $u_n \in L^p(\mathbb R^{d+1})$ of equation
$
(\lambda+\partial_t+(-\Delta)^{\frac{\alpha}{2}} + b_n \cdot \nabla)u_n=f_n.
$
Moreover, $$u_n=\Theta_p(b_n)f_n,$$ where $\Theta_p(b_n)f_n$ coincides with the Duhamel series representation for $u_n$. This can be seen directly by iterating the Dumahel formula for $u_n$ and showing that the remainder of the series goes to zero sufficiently rapidly, using the fact that, clearly, $u_n \in L^p(\mathbb R^{d+1})$, and applying the uniform in $n$ estimates on the norms of $R_p(b_n)$, $Q_p(b_n)$ obtained above. (Alternatively, one can argue as in \cite{Ki_super}, i.e.\,prove that the operator-valued function $\lambda \mapsto \Theta_p(b_n,\lambda)$ satisfies the resolvent identity,then  use the fact that, by a classical result, identity $u_n=\Theta_p(b_n)f_n$ holds for all $\lambda \geq \lambda_0(\|b_n\|_\infty)$, and then finally apply the resolvent identity to extend $u_n=\Theta_p(b_n)f_n$ to all $\lambda_{d,\alpha,p,q} < \lambda < \lambda_0(\|b_n\|_\infty)$.)

By the Dominated convergence theorem,
\begin{equation*}
R_p(\bar{b}_n) \rightarrow R_p(b), \;\; Q_p(\bar{b}_n) \rightarrow Q_p(b) \quad \text{ strongly in $L^p(\mathbb R^{d+1})$},
\end{equation*}
while by \eqref{10_ineq}, combined with H\"{o}lder's inequality and the parabolic Sobolev embedding,
\begin{equation*}
R_p(\bar{b}_n) -R_p(b_n) \rightarrow 0, \;\; Q_p(\bar{b}_n)-Q_p(b_n) \rightarrow 0 \quad \text{ strongly in $L^p(\mathbb R^{d+1})$}.
\end{equation*}
Thus,
\begin{equation}
\label{R_Q_conv}
R_p(b_n) \rightarrow R_p(b), \;\; Q_p(b_n) \rightarrow Q_p(b) \quad \text{ strongly in $L^p(\mathbb R^{d+1})$},
\end{equation}
Also, by our assumption on $f_n$ and $f$,
$$
(\lambda+\partial_t  +(- \Delta)^{\frac{\alpha}{2}})^{(-1+\frac{1}{\alpha})\frac{1}{p'}}f_n \rightarrow (\lambda+\partial_t  +(- \Delta)^{\frac{\alpha}{2}})^{(-1+\frac{1}{\alpha})\frac{1}{p'}}f \quad \text{ in $L^p(\mathbb R^{d+1})$}. 
$$
Therefore,
\begin{equation}
\label{u_conv}
u_n=\Theta_p(b_n)f_n \rightarrow u:=\Theta_p(b)f \quad \text{ in  }\mathbb{W}_\alpha^{\frac{2}{\alpha}+\frac{\alpha-1}{\alpha}\frac{2}{p},p},
\end{equation}
as claimed.

\smallskip

(\textit{ii}) Let us take $p=2$ in (\textit{i}), so $u \in \mathbb{W}_\alpha^{\frac{2}{\alpha}+\frac{\alpha-1}{\alpha},2}$. Let us show that this is a weak solution of \eqref{eq1} in the sense of definition \eqref{weak_sol_def}. Multiplying 
$(\lambda + \partial_t + (- \Delta)^{\frac{\alpha}{2}} + b_n \cdot \nabla)u_n=f_n$, $n=1,2,\dots$, by test function $$\varphi=(\lambda-\partial_t+(-\Delta)^{\frac{\alpha}{2}})^{-\frac{1}{2}+\frac{1}{2\alpha}}(\lambda+\partial_t+(-\Delta)^{\frac{\alpha}{2}})^{\frac{1}{2}+\frac{1}{2\alpha}}\eta, \quad \eta \in C_c^\infty(\mathbb R^{d+1})$$ and integrating over $\mathbb R^{d+1}$, we obtain
\begin{align}
\langle (\lambda+\partial_t&+(-\Delta)^{\frac{\alpha}{2}})^{\frac{1}{2}+\frac{1}{2\alpha}}u_n, (\lambda + \partial_t+(-\Delta)^{\frac{\alpha}{2}})^{\frac{1}{2}+\frac{1}{2\alpha}}\eta \rangle \notag \\[2mm]
&+ \langle R_2(b_n) (\lambda+\partial_t+(-\Delta)^{\frac{\alpha}{2}})^{\frac{1}{2}+\frac{1}{2\alpha}}u_n,Q_2^\ast(b_n)(\lambda+\partial_t+(-\Delta)^{\frac{\alpha}{2}})^{\frac{1}{2}+\frac{1}{2\alpha}} \eta \rangle \notag \\[2mm]
&=\langle f_n,(\lambda-\partial_t+(-\Delta)^{\frac{\alpha}{2}})^{-\frac{1}{2}+\frac{1}{2\alpha}}(\lambda+\partial_t+(-\Delta)^{\frac{\alpha}{2}})^{\frac{1}{2}+\frac{1}{2\alpha}}\eta \rangle, \label{weak_id}
\end{align}
where $Q_2^\ast(b_n)=|b_n|^{\frac{1}{2}} (\lambda - \partial_t  +(- \Delta)^{\frac{\alpha}{2}})^{(-1+\frac{1}{\alpha})\frac{1}{2}} \in \mathcal B(L^2(\mathbb R^{d+1}))$.

1.~In view of \eqref{u_conv},
\begin{align*}
\langle (\lambda+\partial_t+(-\Delta)^{\frac{\alpha}{2}})^{\frac{1}{2}+\frac{1}{2\alpha}}u_n, & (\lambda + \partial_t+(-\Delta)^{\frac{\alpha}{2}})^{\frac{1}{2}+\frac{1}{2\alpha}}\eta \rangle \\
& \rightarrow \langle (\lambda+\partial_t+(-\Delta)^{\frac{\alpha}{2}})^{\frac{1}{2}+\frac{1}{2\alpha}}u, (\lambda + \partial_t+(-\Delta)^{\frac{\alpha}{2}})^{\frac{1}{2}+\frac{1}{2\alpha}}\eta \rangle \quad (n \rightarrow \infty).
\end{align*}

2.~Next,
\begin{align*}
\langle R_2(b_n) & (\lambda+\partial_t+(-\Delta)^{\frac{\alpha}{2}})^{\frac{1}{2}+\frac{1}{2\alpha}}u_n ,Q_2^\ast(b_n)(\lambda+\partial_t+(-\Delta)^{\frac{\alpha}{2}})^{\frac{1}{2}+\frac{1}{2\alpha}} \eta \rangle \\
& = \langle R_2(b_n)(\lambda+\partial_t+(-\Delta)^{\frac{\alpha}{2}})^{\frac{1}{2}+\frac{1}{2\alpha}} (u_n-u),Q_2^\ast(b_n)(\lambda+\partial_t+(-\Delta)^{\frac{\alpha}{2}})^{\frac{1}{2}+\frac{1}{2\alpha}}\eta\rangle \\  
& + \langle R_2(b_n)(\lambda+\partial_t+(-\Delta)^{\frac{\alpha}{2}})^{\frac{1}{2}+\frac{1}{2\alpha}} u,(Q_2^\ast(b_n)-Q_2^\ast(b))(\lambda+\partial_t+(-\Delta)^{\frac{\alpha}{2}})^{\frac{1}{2}+\frac{1}{2\alpha}}\eta\rangle \\
& + \langle R_2(b_n)(\lambda+\partial_t+(-\Delta)^{\frac{\alpha}{2}})^{\frac{1}{2}+\frac{1}{2\alpha}} u,Q_2^\ast(b)(\lambda+\partial_t+(-\Delta)^{\frac{\alpha}{2}})^{\frac{1}{2}+\frac{1}{2\alpha}}\eta\rangle.
\end{align*}
By \eqref{u_conv} and $Q_2^\ast(b_n) \rightarrow Q_2^\ast(b)$ strongly in $L^2(\mathbb R^{d+1})$ (proved using the same argument as in \eqref{R_Q_conv}), we obtain that the first two terms in the RHS tend to $0$ as $n \rightarrow \infty$. By \eqref{R_Q_conv}, the last term tends to $\langle R_2(b)(\lambda+\partial_t+(-\Delta)^{\frac{\alpha}{2}})^{\frac{1}{2}+\frac{1}{2\alpha}} u,Q_2^\ast(b)(\lambda+\partial_t+(-\Delta)^{\frac{\alpha}{2}})^{\frac{1}{2}+\frac{1}{2\alpha}}\eta\rangle$.

3.~Since $f_n \rightarrow f$ in  $\mathbb{W}_\alpha^{-1+\frac{1}{\alpha},2}$, we have
\begin{align*}
& \langle f_n,(\lambda-\partial_t+(-\Delta)^{\frac{\alpha}{2}})^{-\frac{1}{2}+\frac{1}{2\alpha}}(\lambda+\partial_t+(-\Delta)^{\frac{\alpha}{2}})^{\frac{1}{2}+\frac{1}{2\alpha}}\eta \rangle \\
& = \langle (\lambda + \partial_t+(-\Delta)^{\frac{\alpha}{2}})^{-\frac{1}{2}+\frac{1}{2\alpha}}f_n,(\lambda+\partial_t+(-\Delta)^{\frac{\alpha}{2}})^{\frac{1}{2}+\frac{1}{2\alpha}}\eta \rangle \\
& \rightarrow \langle (\lambda + \partial_t+(-\Delta)^{\frac{\alpha}{2}})^{-\frac{1}{2}+\frac{1}{2\alpha}}f,(\lambda+\partial_t+(-\Delta)^{\frac{\alpha}{2}})^{\frac{1}{2}+\frac{1}{2\alpha}}\eta \rangle \quad (n \rightarrow \infty).
\end{align*}

Applying 1-3 in \eqref{weak_id}, we obtain that $u$ is a weak solution to \eqref{eq1} in the sense of \eqref{weak_sol_def}. 

\smallskip

Let us prove  uniqueness. Let $v\in \mathbb{W}_\alpha^{\frac{2}{\alpha}+\frac{\alpha-1}{\alpha},2}$ be another weak solution. Put
\begin{align*}
\tau[v,\eta]:=\langle (\lambda+\partial_t&+(-\Delta)^{\frac{\alpha}{2}})^{\frac{1}{2}+\frac{1}{2\alpha}}v, (\lambda + \partial_t+(-\Delta)^{\frac{\alpha}{2}})^{\frac{1}{2}+\frac{1}{2\alpha}}\eta \rangle \notag \\[2mm]
&+ \langle R_2(b_n) (\lambda+\partial_t+(-\Delta)^{\frac{\alpha}{2}})^{\frac{1}{2}+\frac{1}{2\alpha}}v,Q_2^\ast(b_n)(\lambda+\partial_t+(-\Delta)^{\frac{\alpha}{2}})^{\frac{1}{2}+\frac{1}{2\alpha}} \eta \rangle,
\end{align*}
where $\eta \in C_c^\infty(\mathbb R^{d+1})$. We have
$$
|\langle R_2(b)(\lambda+\partial_t-\Delta)^{\frac{3}{4}}v,Q^\ast_2(b)(\lambda+\partial_t-\Delta)^{\frac{3}{4}}\eta\rangle| \leq c\|v\|_{\mathbb W_\alpha^{1+\frac{1}{\alpha},2}}\|\eta\|_{\mathbb W_\alpha^{1+\frac{1}{\alpha},2}}
$$
where $c<1$ by our assumption on $b$. We extend $\tau[v,\eta]$ to $\eta \in \mathbb W_\alpha^{1+\frac{1}{\alpha},2}$ by continuity. Now, we have $\tau[v-u,\eta]=0$, where $u$ is the weak solution constructed above, so it suffices to choose $\eta=v-u$ to arrive at $
0=\tau[v-u,v-u] \geq (1-c)\|v\|_{\mathbb W_\alpha^{1+\frac{1}{\alpha},2}}^2,
$ hence $v=u$. \hfill \qed

\bigskip

\section{Proof of Theorem \ref{thm2}}

Let us define $U^{t,r}g:=v(t)$ ($t \geq r$), $g \in C_\infty(\mathbb R^d) \cap W^{1,p}(\mathbb R^d)$ where $v(t)$ is given by \eqref{v_repr}. 
Since $U^{t,r}_n$ are $L^\infty$ contractions, it suffices to prove 
\begin{equation}
\label{v_conv}
Ug=C_b(D_T,C_\infty(\mathbb R^d))\mbox{-}\lim_{n}U_n g,
\end{equation}
where $D_T:=\{(s,t) \in \mathbb R^2 \mid 0 \leq s \leq t \leq T\}$,
and then extend operators $U^{t,r}$ by continuity to $g \in C_\infty(\mathbb R^d)$. The reproduction property of $U^{t,r}$ and the preservation of positivity will follow from the corresponding properties of $U^{t,r}_n$.

\smallskip

Let us prove \eqref{v_conv}. Put $v_n:=U^{t,r}_n g$.
We have
$$
v_n=(\lambda+\partial_t+(-\Delta)^{\frac{\alpha}{2}})^{-1}\delta_{r}g - (\lambda+\partial_t+(-\Delta)^{\frac{\alpha}{2}})^{-\frac{1}{\alpha}+(-1+\frac{1}{\alpha})\frac{1}{p}}Q_p(b_n) (1+R_p(b_n)Q_p(b_n))^{-1}G_p(b_n)  S_p g.
$$
This is the usual Duhamel series representation for $v_n$.
We know from the proof of Theorem \ref{thm1} that operators $Q_p(b_n)$, $R_p(b_n)$, $G_p(b_n)$ are bounded on $L^p(\mathbb R^{d+1})$ with operator norms independent of $n$. In turn, operator $S_p$
satisfies
$$
\|S_pg\|_{L^p(\mathbb R^{d+1})} \leq C_{d,\alpha,p} \|\nabla g\|_{L^p(\mathbb R^d)}.
$$
(Indeed, taking for brevity $r=0$, we have by definition
\begin{align*}
S_p g(t,x)=\mathbf{1}_{t > 0}e^{-\lambda t} \int_{\mathbb R^d} \nabla_x q_\gamma (t,x-y)g(y) dy  = \mathbf{1}_{t > 0}e^{-\lambda t} \int_{\mathbb R^d} q_\gamma (t,x-y) \nabla_y g(y) dy.
\end{align*}
By \eqref{ul_bounds}, 
\begin{align*}
|S_p g(t,x)| \leq C\mathbf{1}_{t > 0}e^{-\lambda t}t^{-\frac{1}{p}+\frac{1}{\alpha p}} \int_{\mathbb R^d} p_1 (t,x-y) |\nabla_y g(y)| dy,
\end{align*}
and so
$$
\|S_p g\|_{L^p(\mathbb R^{d+1})}^p \leq C \int_{\mathbb R} \mathbf{1}_{t > 0}\|S_p(t)\|_{L^p(\mathbb R^{d})}^p dt \leq C\int_0^\infty e^{-\lambda p t} t^{(-\frac{1}{p}+\frac{1}{\alpha p})p} dt \|\nabla g\|_{L^p(\mathbb R^d)}^p,
$$
where $(-\frac{1}{p}+\frac{1}{\alpha p})p=-1+\frac{1}{\alpha}>-1$ so the integral in time converges.)

Clearly, $(\lambda+\partial_t-\Delta)^{-1}\delta_{r}g \in C_b([r,\infty[,C_\infty(\mathbb R^d))$.
Thus, to prove \eqref{v_conv}, it remains to note that $Q_p(b_n) \rightarrow Q_p(b)$, $R_p(b_n) \rightarrow R_p(b)$ and $G_p(b_n) \rightarrow G_p(b)$ strongly in $L^{p}(\mathbb R^{d+1})$, which was established in the proof of Theorem \ref{thm1}, so that by the parabolic Sobolev embedding \eqref{sob_emb0}, \eqref{sob_emb}, since $p>d+1$, 
\begin{align*}
&(\lambda+\partial_t+(-\Delta)^{\frac{\alpha}{2}})^{-\frac{1}{\alpha}+(-1+\frac{1}{\alpha})\frac{1}{p}}Q_p(b_n) (1+R_p(b_n)Q_p(b_n))^{-1}G_p(b_n)  S_p g \\
&\rightarrow (\lambda+\partial_t+(-\Delta)^{\frac{\alpha}{2}})^{-\frac{1}{\alpha}+(-1+\frac{1}{\alpha})\frac{1}{p}} Q_p(b) (1+R_p(b)Q_p(b))^{-1}G_p(b)  S_p g
\end{align*}
in $C_\infty(\mathbb R^{d+1})$ as $n \rightarrow \infty$.
The sought convergence \eqref{v_conv} follows. \hfill \qed

\bigskip

\section{Probability measures $\{\mathbb P_x\}_{x \in \mathbb R^d}$ and Krylov bound \eqref{krylov}}

\label{krylov_sect}

We continue the discussion started in section \textbf{a)} of the introduction. We make a few simplifying assumptions, i.e.\, we ignore the representation of $b$ as a sum of bounded and unbounded vector fields, and will require smallness of the Morrey norm of $|b|^{1/(\alpha-1)}$. Also, we assume that $b$ has compact support. (These assumptions are not crucial and can be removed with a few additional efforts.)

\medskip

1.~Let us first discuss the relationship between the SDE for the stable process with bounded smooth drift $b$ and the corresponding parabolic equation. So, let $b=b_n$ be defined by \eqref{b_n}.
By the classical theory, the unique weak solution  $\mathbb P^n_x$ of SDE
$$
\omega_t=x-\int_0^b b_n(s,\omega_s)ds + Z_t, \quad t \geq 0, \quad \text{$Z_t$ is isotropic $\alpha$-stable process}
$$
satisfies
$$
\mathbb E_{\mathbb P^n_x}[e^{-\lambda r}f(\omega_r)]=v_n(0,x),
$$
and
\begin{equation}
\label{id_nonhom}
\mathbb E_{\mathbb P^n_x}\bigg[\int_0^r \int_{\mathbb R^d}e^{-\lambda s}F(s,y)dyds\bigg]=w_n(0,x),
\end{equation}
where $v_n$ and $w_n$ are the classical solutions of terminal-value problems for the backward Kolmogorov equations in $t \in ]-\infty,r]$:
\begin{equation}
(\lambda-\partial_t + (-\Delta)^{\frac{\alpha}{2}} + b_n \cdot \nabla )v_n=0, \quad  v(r,\cdot)=f(\cdot),\label{eq10}
\end{equation}
and
\begin{equation}
(\lambda-\partial_t + (-\Delta)^{\frac{\alpha}{2}} + b_n \cdot \nabla )w_n=F, \quad  w(r,\cdot)=0 \label{eq11}
\end{equation}
(functions $f=f(x)$ and $F=F(t,x)$ are, say, smooth with compact supports).
These parabolic equations are of the form considered in Theorems \ref{thm1}, \ref{thm2}, up to reversing the direction of time. The latter does not affect the proofs of Theorems \ref{thm1}, \ref{thm2} since it does not change the Morrey norm of the unbounded part of the drift.

\medskip

2.~In order to handle general $b$ satisfying the assumptions of Theorems \ref{thm1}, \ref{thm2}, we will reverse the relationship between $\mathbb P_x$ and  the parabolic equations outlined above. That is, we will use the regularity results for parabolic equations in order construct probability measures $\mathbb P_x$, $x \in \mathbb R^d$, and to show that they satisfy Krylov bound \eqref{krylov}.

In detail, applying Theorem \ref{thm2} to equation \eqref{eq10} (after reversing the direction of time), we obtain a backward Feller evolution family $\{P^{t,r}\}_{t \leq r}$. By a standard result (see e.g.\,\cite{GV}),  there exist measures $\mathbb P_x$, $x \in \mathbb R^d$ on the space of c\`{a}dl\`{a}g paths $\omega_t$
such that
$$
\mathbb E_{\mathbb P_x}\big[e^{-\lambda r}f(\omega_r)\big]=P^{0,r}f(x), \quad r \geq 0, \quad \text{ for all } f \in C_\infty(\mathbb R^d).
$$
These $\{\mathbb P_x\}_{x \in \mathbb R^d}$ are probability measures, see Remark \ref{prop_rem}. 

\medskip

Now, we obtain from Theorem \ref{thm1}:

\begin{corollary}
\label{krylov_thm}
Assume that $|b|^{\frac{1}{\alpha-1}} \in E_{1+\varepsilon}$ and $|F|^{\frac{1}{\alpha-1}} \in E_{1+\varepsilon}$ for some $0<\varepsilon<d+\alpha-1$ (we are interested in $\varepsilon$ close to $0$) and  $$
\||b|^{\frac{1}{\alpha-1}}\|_{E_{1+\varepsilon}}<c_{d,\alpha,p,1+\varepsilon}
$$
for some  $p>d+1$. Also, let us assume that $F$ has compact support. Then, for every $x \in \mathbb R^d$,
\begin{align}
\mathbb E_{\mathbb P_x} \int_0^t |F(s,\omega_s)| ds & \leq C \|F\|_{L^1([0,t] \times \mathbb R^d)}^{\frac{1}{p}} \label{est2} \\
& \leq \hat{C}t^\frac{1}{q'}\|F\|_{L^q([0,t] \times \mathbb R^d)}<\infty, \quad 0\leq t \leq 1, \notag
\end{align}
where $\hat{C}:=C|\sprt F|^{\frac{1}{q'}}<\infty$, constant $C$  depends only on $\||b_{\mathfrak s}|^{\frac{1}{\alpha-1}}|\|_{E_{1+\varepsilon}}$, $\||F|^{\frac{1}{\alpha-1}}|\|_{E_{1+\varepsilon}}$, $\alpha$, $d$, $\varepsilon$ and $p$.
\end{corollary}

\begin{remark}
At the first sight, inequality \eqref{est2} is inhomogeneous in $F$. However, multiplying $F$ by a constant changes Morrey norm $\||F|^{1/(\alpha-1)}|\|_{E_{1+\varepsilon}}$ and hence changes constant $C$.
\end{remark}

\begin{proof} Step 1.~First, let vector field $b$ and function $F$ be bounded. Then one sees that the proof of 
\begin{align*}
\mathbb E_{\mathbb P_x} \int_0^t |F(s,\omega_s)| ds & \leq C \|F\|_{L^1([0,t] \times \mathbb R^d)}^{\frac{1}{p}},
\end{align*}
amounts to proving, taking into account identity \eqref{id_nonhom} and reversing the direction of time,
the estimate 
\begin{equation}
\label{est3} 
\|w\|_{L^\infty([0,t]\times \mathbb R^d)} \leq C_t\biggl(\int_0^t \int_{\mathbb R^{d}}|F| dyds\biggr)^{\frac{1}{p}}, \quad p>d+1,\;\;C:=\sup_{t \in [0,1]}C_t<\infty
\end{equation}
for solution $w$ of initial-value problem
\begin{equation}
\label{cp1}
\left\{
\begin{array}{l}
\partial_t w + (- \Delta)^{\frac{\alpha}{2}}w + b \cdot \nabla w=|F|, \\
w(0,\cdot)=0,
\end{array}
\right.
\end{equation}
Let us rewrite \eqref{cp1} in the form of \eqref{eq1}, i.e.\,put $w=e^{\lambda t} u$, where
\begin{equation}
\label{eq2}
\lambda u + \partial_t u + (- \Delta)^{\frac{\alpha}{2}}u + b \cdot \nabla u=e^{-\lambda t}|F|\mathbf{1}_{t>0}  \quad \text{ on } \mathbb R^{d+1}.
\end{equation}
Put $\tilde{F}:=e^{-\lambda t}|F|\mathbf{1}_{t>0}$.
So, by Theorem \ref{thm1}(\textit{i}) and the boundedness of operator $Q_p$,
\begin{align}
u = & (\lambda+\partial_t  +(- \Delta)^{\frac{\alpha}{2}})^{-1}f \notag \\
& - (\lambda+\partial_t  +(- \Delta)^{\frac{\alpha}{2}})^{-\frac{1}{\alpha}+(-1+\frac{1}{\alpha})\frac{1}{p}}Q_p(b) (1+R_p(b)Q_p(b))^{-1}R_p(b) Q_p(\tilde{F}) \tilde{F}^{\frac{1}{p}}, \label{emb_9}
\end{align}
for all $\lambda > \lambda_{\alpha,d,p,q}$,
where, clearly, $\tilde{F}  \in L^p(\mathbb R^{d+1})$.

\begin{remark}
Let us emphasize rather a quite natural but still nice aspect of representation \eqref{u_repr} of solution $u$ in Theorem \ref{thm1}(\textit{i}): when we apply it to parabolic equation \eqref{eq2}, we write
$$
 (\lambda+\partial_t  +(- \Delta)^{\frac{\alpha}{2}})^{(-1+\frac{1}{\alpha})\frac{1}{p'}} \tilde{F}= Q_p(\tilde{F})\tilde{F},
$$
i.e.\,the right-most ``free term'' in representation \eqref{u_repr} becomes a bounded on $L^p(\mathbb R^{d+1})$ operator $Q_p(\tilde{F})$  applied to a function in $L^p(\mathbb R^{d+1})$.
In other words,  $$\tilde{F} \in \mathbb{W}_\alpha^{(-1+\frac{1}{\alpha})\frac{2}{p'},p}(\mathbb R^{d+1}),$$ i.e.\,\eqref{eq2} fits in the setting of Theorem \ref{thm1}.
\end{remark}

Since $p>d+1$, representation \eqref{emb_9} and the parabolic Sobolev embedding give us \eqref{est3} and hence \eqref{est2}.

\medskip

Step 2.~It remains to get rid of the smoothness assumptions on $b$ and $F$. Since constant 
$C$ in \eqref{est2} depends on the Morrey norms of $|b|^{1/(\alpha-1)}$ and $|F|^{1/(\alpha-1)}$, but not on their boundedness or smoothness, it is not difficult to do.

So, let $b$ and $F$  be as in  Corollary \ref{krylov_thm}, i.e.\,in general locally unbounded. 
Define $$F_m:=\gamma_{\epsilon_m}\ast \mathbf{1}_m F, \quad \text{$\mathbf{1}_m$ is the indicator of $\{|t| \leq m, |x| \leq m, |F(t,x)| \leq m \}$}
$$ and $\gamma_{\epsilon_m}$ is the Friedrichs mollifier on $\mathbb R^{d+1}$. Selecting $\epsilon_m \downarrow 0$ sufficiently rapidly, we have 
$$\||F_m|^{\frac{1}{\alpha-1}}\|_{E_{q}} \leq (1+m^{-1})\||F|^{\frac{1}{\alpha-1}}\|_{E_{q}}.$$ Then, by Step 1,
\begin{align*}
\mathbb E_{\mathbb P^n_x} \int_0^t |F_m(s,\omega_s)| ds & \leq C \|F_m\|_{L^1([0,t] \times \mathbb R^d)}^{\frac{1}{p}},
\end{align*}
where $C$ is independent of $n$, $m$.
We now pass to the limit $n \rightarrow \infty$ in the previous estimate using the weak convergence of $\mathbb P_x^n$ to $\mathbb P_x$, which is a direct consequence of Theorem \ref{thm2}(\textit{i}). Next, applying Fatou's lemma, we pass to the limit in $m \rightarrow \infty$. (It should be added that we only have convergence $F_m$ to $F$ only a.e.\,on $\mathbb R^{d+1}$, but this does not cause any difficulty to us since $\mathbb E_{\mathbb P_x} \int_0^t |F(s,\omega_s)| ds$ does not sense a modification of $F$ on a measure zero set $C \subset \mathbb R^{d+1}$. This can be seen by selecting instead of $F$ the indicator function $\mathbf{1}_{C_\varepsilon}$ of an $\varepsilon$-neighbourhood $C_\varepsilon$ of $C$, mollifying it, and then applying already established for bounded smooth $F$ estimate \eqref{est2} whose right-hand side tends to zero as $\varepsilon \downarrow 0$.)
\end{proof}

\begin{remark}
\label{prop_rem}
The fact that $\{\mathbb P_x\}_{x \in \mathbb R^d}$ are probability measures can be proved as follows. Since $b$ has compact support, we may further assume without loss of generality that $\sprt b \subset [0,1] \times \bar{B}_1(0)$. Clearly, $\mathbb P_x^n$ are probability measures, so we need to rule out the loss of probability as we take $n \rightarrow \infty$ in Theorem \ref{thm2}. To this end, it suffices to show that, for a fixed $x_0 \in \mathbb R^d$,   
\begin{align}
v(t,x_0)=& (\lambda+\partial_t+(-\Delta)^{\frac{\alpha}{2}})^{-1}\delta_{r}g(t,x_0) \notag \\
& - \bigl[(\lambda+\partial_t+(-\Delta)^{\frac{\alpha}{2}})^{-\frac{1}{\alpha}+(-1+\frac{1}{\alpha})\frac{1}{p}}Q_p(b_n) (1+R_p(b_n)Q_p(b_n))^{-1}G_p(b_n)  S_p(b_n) g\bigr](t,x_0) \label{v}
\end{align}
can be made arbitrarily small uniformly in $t \in ]0,1]$ and $n \geq 1$, by selecting initial function $g \in C_c^\infty(\mathbb R^d)$, $g=0$ in $B_R(0)$, $0 \leq g \leq 1$, with $R$ sufficiently large (so, we let this function $g$ to be equal to $1$ in arbitrarily large neighbourhood of $\mathbb R^d - \bar{B}_R(0)$). This is evidently true for the first term $(\lambda+\partial_t+(-\Delta)^{\frac{\alpha}{2}})^{-1}\delta_{r}g(t,x_0)$ in \eqref{v} due to the ``separation property'' of the integral kernel of operator $(\lambda+\partial_t+(-\Delta)^{\frac{\alpha}{2}})^{-1}\delta_{r}$, since $x_0$ is far away from the support of $g$ if radius $R$ is chosen sufficiently large. 

Since $b$ has compact support (in $[0,1] \times \bar{B}_1(0)$), the same argument applies to the second term in \eqref{v}. 
Namely, by \eqref{grad_bd},
we have
\begin{equation}
\label{GS}
|G_p(b_n)S_p(b_n)g(t,x)| \leq C|b_n(t,x)|^{\frac{1}{p}}e^{-\lambda t}\int_{\mathbb R^d}p_{2-\frac{2}{\alpha}}(t,x-y)g(y)dy.
\end{equation}
Selecting $R$ sufficiently large we may assume that $x$ (in $B_1(0)$) and $y$ (in $\mathbb R^d - \bar{B}_R(0)$) are far away from each other. So, we have by \eqref{ul_bounds} $p_{2-\frac{2}{\alpha}}(t,x-y) \leq t^{1-\frac{1}{\alpha}}|x-y|^{-d-\alpha}$, hence the integral in the right-hand side of \eqref{GS}
can be made arbitrarily small uniformly in $x$ and $t \in [0,1]$ by selecting $R$ large. In turn, the first multiple $C|b_n|^{1/p}$ is, clearly, bounded in $L^p(\mathbb R^{d+1})$ uniformly in $n$. To summarize, we can make $\|G_p(b_n)S_p(b_n)g\|_{L^p(\mathbb R^{d+1})}$ arbitrarily small uniformly in $n$ by selecting $R$ sufficiently large. Hence, by the Sobolev embedding property, the second term in \eqref{v} can be made arbitrarily small.

In \cite{KM}, we proved that  $\{\mathbb P_x\}_{x \in \mathbb R^d}$ are probability measures by means of a more sophisticated argument that uses weights, but at the same time in \cite{KM} we did not require from a (time-homogeneous) $b$ to have compact support.

\end{remark}

\bigskip

\section{A priori regularity estimates for McKean-Vlasov equations}

\label{mv_sect}

We continue the discussion from section \textbf{b)} of the introduction. We are looking for a priori estimates on the solution of initial-value problem \eqref{mv_eq} for McKean-Vlasov equation.

\medskip

1.~Let us first state a straightforward a priori counterpart of Theorem \ref{thm1} for the initial-value problem forward Kolmogorov equation
\begin{equation}
\label{eq_f}
\lambda \eta + \partial_t \eta + (- \Delta)^{\frac{\alpha}{2}}\eta - {\rm div\,}(b\eta)=0, \quad \eta(0,\cdot)=h(\cdot),
\end{equation}
where, recall, the initial distribution $h$ and interaction kernel $b$ are bounded and smooth,
$$
0 \leq h, \quad \int_{\mathbb R^d}h=1,
$$
but the estimates on $\rho$ that we are looking for should not depend on the boundedness or smoothness of $h$, $\rho$. Namely, \textit{under the assumption on $b$ of Theorem \ref{thm1},
for every $1<p<\infty$, if
$$\||b_{\mathfrak s}|^{\frac{1}{\alpha-1}}\|_{E_{q}} < c_{d,\alpha,p,q}$$ 
then for all $\lambda > \lambda_{d,\alpha,p,q}>0$
\begin{align}
 \eta & =  (\lambda+\partial_t  +(- \Delta)^{\frac{\alpha}{2}})^{-1}\delta_0\,h \notag \\
& + \nabla (\lambda+\partial_t  +(- \Delta)^{\frac{\alpha}{2}})^{-\frac{1}{p}-\frac{1}{\alpha p'}}Q_p (1+R_pQ_p)^{-1}G_p (\lambda+\partial_t  +(- \Delta)^{\frac{\alpha}{2}})^{-\frac{1}{p'}-\frac{1}{\alpha p}}\delta_0\, h, \label{rho_repr}
\end{align}
where $Q_p$, $R_p$ and $G_p$ are bounded operators on $L^p(\mathbb R^{d+1})$ defined in Theorems \ref{thm1}, \ref{thm2}, and 
$\|R_pQ_p\|_{p \rightarrow p} <1$, so
\begin{equation}
\label{eta_est}
\|\eta\|_{\mathbb W^{\frac{2}{p}(1-\frac{1}{\alpha}),p}_\alpha(\mathbb R^{d+1})} \leq C_{d,\alpha,p,q} \big\|\big(\lambda+\partial_t+(-\Delta)^{\frac{\alpha}{2}}\big)^{-\frac{1}{p'}-\frac{1}{\alpha p}}\delta_0\,h\big\|_{L^p(\mathbb R^{d+1})}.
\end{equation}
}
(It is not difficult to see that if we chose $b_{\mathfrak b}=0$, so that $b=b_{\mathfrak s}$, then above we can consider any $\lambda>0$.)
\begin{remark}
A key aspect of what is written above is that  constants $c_{d,\alpha,p,q}$ and $C_{d,\alpha,p,q}$  do not depend on boundedness or smoothness of $b$ and $h$. In fact, it is not difficult to state and prove complete analogue Theorem \ref{thm1} for \eqref{eq_f} (except the uniform convergence, of course) for locally unbounded $b$ and $h$, in which case estimate \eqref{eta_est} becomes a posteriori estimate.
\end{remark}

2.~Let $b:\mathbb R^d \rightarrow \mathbb R^d$ have sufficiently small elliptic Morrey norm, i.e.
$$
\||b|^{\frac{1}{\alpha-1}}\|_{M_{1+\varepsilon}} \leq c_{d,\alpha,p,1+\varepsilon} \qquad \text{(see \eqref{morrey_elliptic})}.
$$
Let us rewrite McKean-Vlasov equation \eqref{mv_eq} in form \eqref{eq_f}. Put $\xi(t,x):=e^{-\lambda t} \rho(t,x)$, where $\rho$ is the solution of \eqref{mv_eq}. Then
$$
\lambda\xi + \partial_t \xi + (-\Delta)^{\frac{\alpha}{2}} \xi - {\rm div\,}\big[ \tilde{b} \xi \big]=0, \quad  \xi(0,\cdot)=h(\cdot),
$$
where 
\begin{equation}
\label{tilde_b_b}
\||\tilde{b}|^{\frac{1}{\alpha-1}}\|_{E_{1+\varepsilon}} \leq \||b|^{\frac{1}{\alpha-1}}\|_{M_{1+\varepsilon}}.
\end{equation}
Thus, we obtain from \eqref{eta_est} the following estimates.

\begin{corollary}
\label{cor_mv}
Given any $1<p<\infty$, if $\||b|^{\frac{1}{\alpha-1}}\|_{M_{1+\varepsilon}} <c_{d,\alpha,p,1+\varepsilon}$, then
\begin{equation}
\label{mv_emb0_}
\|\xi\|_{\mathbb W^{\frac{2}{p}(1-\frac{1}{\alpha}),p}_\alpha(\mathbb R^{d+1})} \leq C_\lambda \big\|\big(\lambda+\partial_t+(-\Delta)^{\frac{\alpha}{2}}\big)^{-\frac{1}{p'}-\frac{1}{\alpha p}}\delta_0\,h\big\|_{L^p(\mathbb R^{d+1})}, \quad \lambda>0,
\end{equation}
where $\delta_0$ is the delta-function in the time variable concentrated at $t=0$, see \eqref{delta_def}, constant $C=C(d,\alpha,\varepsilon,p)$ is independent of boundedness or smoothness of $b$.
In particular,
\begin{equation}
\label{mv_emb_}
\|\xi\|_{\mathbb W^{\frac{2}{p}(1-\frac{1}{\alpha}),p}_\alpha(\mathbb R^{d+1})} \leq C_1 \|h\|_{L^r(\mathbb R^d)} \quad \text{ for all }\frac{d}{d+1}p<r \leq p,
\end{equation}
where $C_1=C_1(d,\alpha,\varepsilon,p,1+\varepsilon)$.
\end{corollary}

\begin{proof} 
First, we prove \eqref{mv_emb0_}. In view of \eqref{eta_est}, we only need to show that \eqref{tilde_b_b} holds. Put $e_\lambda(t):=e^{\lambda t}$. We have
\begin{align*}
\||b \ast e_\lambda\xi|^{\frac{1}{\alpha-1}}\|_{E_{1+\varepsilon}} & =\sup_{(t,x) \in \mathbb R^{d+1}, r>0} r \biggl(\frac{1}{|C_r|}\int_{C_r(t,x)}|b \ast e_\lambda \xi|^{\frac{1+\varepsilon}{\alpha-1}}  \biggr)^{\frac{1}{1+\varepsilon}} \\
& = \sup_{(t,x) \in \mathbb R^{d+1}, r>0} r \biggl(\frac{1}{|C_r|}\int_t^{t+r^\alpha}\|b \ast e_\lambda \xi\|^{\frac{1+\varepsilon}{\alpha-1}}_{L^{\frac{1+\varepsilon}{\alpha-1}}(B_r(x))} ds\biggr)^{\frac{1}{1+\varepsilon}} \\
& = \sup_{(t,x) \in \mathbb R^{d+1}, r>0} r \biggl(\frac{1}{|C_r|}\int_t^{t+r^\alpha}\big\|\int_{\mathbb R^d}b(\cdot-y)e^{\lambda s}\xi(s,y)dy\big\|^{\frac{1+\varepsilon}{\alpha-1}}_{L^{\frac{1+\varepsilon}{\alpha-1}}(B_r(x))} ds \biggr)^{\frac{1}{1+\varepsilon}} \\
& \leq \sup_{(t,x) \in \mathbb R^{d+1}, r>0} r \biggl(\frac{1}{|C_r|}\int_t^{t+r^\alpha}\biggl(\int_{\mathbb R^d}\big\|b(\cdot-y)\big\|^{\frac{1+\varepsilon}{\alpha-1}}_{L^{\frac{q}{\alpha-1}}(B_r(x))} e^{\lambda s}\xi(s,y)dy\biggr)^{\frac{1+\varepsilon}{\alpha-1}}  \biggr)^{\frac{1}{1+\varepsilon}} \\
& = \sup_{(t,z) \in \mathbb R^{d+1}, r>0} r \biggl(\frac{1}{|C_r|}\|b\|_{L^{\frac{1+\varepsilon}{\alpha-1}}(B_r(z))}  \int_t^{t+r^\alpha}\biggl(\int_{\mathbb R^d}e^{\lambda s}\xi(s,y)dy\biggr)^{\frac{1+\varepsilon}{\alpha-1}}  \biggr)^{\frac{1}{1+\varepsilon}} \\
& \leq \sup_{z \in \mathbb R^{d}, r>0} r \biggl(\frac{1}{|B_r|}\|b\|_{L^{\frac{1+\varepsilon}{\alpha-1}}(B_r(z))}\biggr)^{\frac{1}{1+\varepsilon}}=\||b|^{\frac{1}{\alpha-1}}\|_{M_{1+\varepsilon}},
\end{align*}
where at the last step we have used a priori estimate $\int_{\mathbb R^d}\xi(s,y)dy \leq e^{-\lambda s}$. 
(We have basically re-proved Young's inequality for Morrey spaces. By the way, this and more general Young's inequalities are discussed in detail in \cite{BT}.)

\medskip

Let us prove \eqref{mv_emb_}. Setting $\gamma:=\frac{2}{p'}+\frac{2}{\alpha p}$, we have by definition
\begin{align*}
\big(\lambda+\partial_t+(-\Delta)^{\frac{\alpha}{2}}\big)^{-\frac{1}{p'}-\frac{1}{\alpha p}}\delta_0\,h(t,x)&=\mathbf{1}_{t>0}t^{\frac{\gamma}{2}-1}e^{-\lambda t}\int_{\mathbb R^d} (\partial_t+(-\Delta)^{\frac{\alpha}{2}})^{-1}(t,x-y)h(y)dy \\
& = \mathbf{1}_{t>0}t^{-\frac{1}{p}+\frac{1}{\alpha p}} e^{-\lambda t}e^{-t(-\Delta)^{\frac{\alpha}{2}}}h(x).
\end{align*}
Therefore, using well-known  bound $\|e^{-t(-\Delta)^{\frac{\alpha}{2}}}\|_{L^r(\mathbb R^d) \rightarrow L^p(\mathbb R^d)} \leq C_{r,p}t^{-\frac{d}{\alpha}(\frac{1}{r}-\frac{1}{p})}$, $1 \leq r \leq p \leq \infty$, we obtain
\begin{align*}
\|\big(\lambda+\partial_t+(-\Delta)^{\frac{\alpha}{2}}\big)^{-\frac{1}{p'}-\frac{1}{\alpha p}}\delta_0\,h\|^p_{L^p(\mathbb R^{d+1})} & \leq \int_{-\infty}^\infty \mathbf{1}_{t>0}t^{-1+\frac{1}{\alpha}} e^{-p\lambda t}\|e^{-t(-\Delta)^{\frac{\alpha}{2}}}h\|^p_{L^p(\mathbb R^d)} dt\\
&  \leq C_{r,p}^p \int_{-\infty}^\infty \mathbf{1}_{t>0}t^{-1+\frac{1}{\alpha}}t^{-\frac{d}{\alpha}(\frac{p}{r}-1)}e^{-p\lambda t} dt \,\|h\|^p_{L^r(\mathbb R^d)},
\end{align*}
where the integral converges if, additionally, $r>\frac{pd}{d+1}$. Thus, \eqref{mv_emb_} follows.
\end{proof}

\appendix

\bigskip

\section{Heat kernel of fractional Laplacian} Given non-negative functions $f$, $g$, we write $f \approx g$ if there exist positive constants $c$ and $C$ such that
$
cf(z) \leq g(z) \leq C g(z)
$
for all $z$ from their domain.

\smallskip

1.~Set $q(t,x):=q_2(t,x)$, $p(t,x)=p_2(t,x)$ (see notations section). The following well-known two-sided estimate holds:
\begin{equation}
\label{qp}
q(t,x) \approx p(t,x).
\end{equation}
Hence, for every  $0 \leq f \in C_c(\mathbb R^{d+1})$,
\begin{align}
(\partial_t + (-\Delta)^{\frac{\alpha}{2}})^{-\frac{\gamma}{2}}f(t,x) & =c_\gamma \int_0^\infty \mu^{-\frac{\gamma}{2}} (\mu+\partial_t + (-\Delta)^{\frac{\alpha}{2}})^{-1}f(t,x) d\mu \notag \\
& = c_\gamma \int_0^\infty \mu^{-\frac{\gamma}{2}} \int_{-\infty}^t \int_{\mathbb R^d} e^{-\mu (t-s)}q(t-s,x-y)f(s,y) dy ds d\mu \notag \\
& \approx c_\gamma \int_0^\infty \mu^{-\frac{\gamma}{2}} \int_{-\infty}^t \int_{\mathbb R^d} e^{-\mu (t-s)}p(t-s,x-y)f(s,y) dy ds d\mu \notag  \\
& = c_\gamma  \int_{-\infty}^t \int_{\mathbb R^d} \biggl(\int_0^\infty \mu^{-\frac{\gamma}{2}}  e^{-\mu (t-s)}d\mu \biggr) p(t-s,x-y)f(s,y) dy ds \notag  \\
& = C_\gamma \int_{-\infty}^t \int_{\mathbb R^d} (t-s)^{\frac{\gamma}{2}-1}  p(t-s,x-y)f(s,y) dy ds \notag  \\
& = C_\gamma \int_{-\infty}^t \int_{\mathbb R^d} p_{\gamma}(t-s,x-y)f(s,y) dy ds, \label{e}
\end{align}
where constants $c_\gamma$, $C_\gamma$ are, of course, independent of $f$.

2.~Another well-known estimate
$$
|\nabla_x q(t,x)| \approx \mathbf{1}_{t>0}|x|\left\{
\begin{array}{ll}
\frac{t}{|x|^{d+2+\alpha}} & \text{ if } |x| \geq t^{\frac{1}{\alpha}}, \\
\frac{1}{t^{\frac{d+2}{\alpha}}} & \text{ if } |x| < t^{\frac{1}{\alpha}},
\end{array}
\right.
$$
(for the proof, see \cite{BJ}) yields
$$
|\nabla_x q(t,x)| \leq \mathbf{1}_{t>0}\left\{
\begin{array}{ll}
t^{1-\frac{1}{\alpha}}|x|^{-d-\alpha} & \text{ if } |x| \geq t^{\frac{1}{\alpha}}, \\
t^{1-\frac{1}{\alpha}}t^{-\frac{d+\alpha}{\alpha}}  & \text{ if } |x| < t^{\frac{1}{\alpha}}.
\end{array}
\right.
$$
Arguing as in Step 1 and using the last estimate, we obtain
\begin{equation}
\label{grad_e}
|\nabla (\partial_t + (-\Delta)^{\frac{\alpha}{2}})^{-\frac{\gamma}{2}}f(t,x)| \leq C'_\gamma \int_{-\infty}^t \int_{\mathbb R^d} p_{\gamma-\frac{2}{\alpha}}(t-s,x-y)f(s,y) dy ds.
\end{equation}

\bigskip

\end{document}